\documentclass[a4paper, leqno, 10pt]{amsart}
\usepackage[utf8]{inputenc}
\usepackage[a4paper, left=90pt, right=90pt, bottom=100pt, top=100pt]{geometry}

\usepackage{amsthm}
\usepackage{amssymb}
\usepackage[american]{babel}
\usepackage{exscale}
\usepackage[mathscr]{euscript}
\usepackage{units}
\usepackage{url}
\usepackage[all,ps,tips,tpic]{xy}
\usepackage{graphicx}
\usepackage{color}
\usepackage{stmaryrd}

\usepackage{tikz} 
\usetikzlibrary{matrix,arrows}
\usepackage{tikz-cd}

\usepackage{amscd}

\numberwithin{equation}{section}

\newtheorem{theorem}{Theorem}[section]
\newtheorem{lemma}[theorem]{Lemma}
\newtheorem{corollary}[theorem]{Corollary}
\newtheorem{proposition}[theorem]{Proposition}
\newtheorem{conjecture}[theorem]{Conjecture}

\theoremstyle{definition}
\newtheorem{definition}[theorem]{Definition}

\theoremstyle{remark}
\newtheorem{claim}{Claim}

\newtheorem*{remark}{Remark}

\title[The dimension of affine Deligne-Lusztig varieties for unramified groups]{The dimension of affine Deligne-Lusztig varieties in the affine Grassmannian of unramified groups}
\author[P. Hamacher]{by Paul Hamacher}
\def \ad {\mathrm{ad}}
\def \dom {\mathrm{dom}}
\def \mmax {\mathrm{max}}

\DeclareMathOperator{\Aut}{Aut}
\DeclareMathOperator{\defect}{def}
\DeclareMathOperator{\diag}{diag}

\DeclareMathOperator{\Int}{Int}
\DeclareMathOperator{\inv}{inv}
\DeclareMathOperator{\Gal}{Gal}
\DeclareMathOperator{\GL}{GL}
\DeclareMathOperator{\Lie}{Lie}
\DeclareMathOperator{\height}{ht}
\DeclareMathOperator{\PGL}{PGL}
\DeclareMathOperator{\suc}{succ}

\DeclareMathOperator{\Res}{Res}
\DeclareMathOperator{\rank}{rk}
\DeclareMathOperator{\val}{val}
\DeclareMathOperator{\vol}{vol}

\renewcommand{\AA}{\mathbb{A}}

\newcommand{\GG}{\mathbb{G}}

\newcommand{\NN}{\mathbb{N}}

\newcommand{\QQ}{\mathbb{Q}}
\newcommand{\RR}{\mathbb{R}}

\newcommand{\ZZ}{\mathbb{Z}}

\newcommand{\calG}{\mathcal{G}}

\newcommand{\calI}{\mathcal{I}}

\newcommand{\calM}{\mathcal{M}}

\newcommand{\calP}{\mathcal{P}}

\newcommand{\calS}{\mathcal{S}}

\newcommand{\calV}{\mathcal{V}}

\newcommand{\fraka}{\mathfrak{a}}

\newcommand{\pot}[1]{ [\hspace{-0,5mm}[ {#1} ]\hspace{-0,5mm}] }
\newcommand{\rpot}[1]{ (\hspace{-0,7mm}( {#1} )\hspace{-0,7mm}) }
\newcommand{\BT}{\mathcal{BT}}
\newcommand{\Grass}{\mathcal{G}\hspace{-0.5mm}r}

\newcommand{\Lra}{\Leftrightarrow}
\newcommand{\teil}[2]{\mbox{{#1}\hspace{0,4mm}({#2}})}
\newcommand{\orbitsum}[1]{\underline{#1}}
\newcommand{\length}[2]{\ell[ {#1} , {#2} ]}
\newcommand{\lengthG}[2]{\ell_G[ {#1}, {#2}]}

\newenvironment{equivarray}{
 
 \begin{displaymath}
 \begin{array}{lc}}{
 \end{array}
 \end{displaymath}}

\begin{document}

 \begin{abstract}
  We calculate the dimension of affine Deligne-Lusztig varieties inside the affine Grassmannian of an arbitrary reductive connected group over a finite field. Our result generalises the dimension formula for split groups, which was determined by the works of G\"ortz. Haines, Kottwitz and Reuman and Viehmann.
 \end{abstract}

\maketitle

 \section{Introduction} 

 Let $k$ be a finite field of characteristic $p$ and let $\overline{k}$ be an algebraic closure of $k$.  We consider a connected reductive group $G$ over $k$. By a theorem of Steinberg, $G$ is quasi-split. Let $k'$ be a finite subfield of $\overline{k}$ such that $G_{k'}$ is split. We fix $S \subset T \subset B \subset G$, where $S$ is a maximal split torus, $T$ a maximal torus which splits over $k'$ and $B$ a Borel subgroup of $G$. Here and in the rest of this article we use the convention that whenever we consider a subgroup of $G$, we automatically assume that it is defined over $k$. We denote $K = G(\overline{k}\pot{t})$ and by $\Grass$ the affine Grassmannian of $G$.

 Denote by $F = k\rpot{t}$, $E = k'\rpot{t}$ and $L = \overline{k}\rpot{t}$ the Laurent series fields. We identify the Galois groups $\Gal(k'/k) = \Gal(E/F) =:I$.  Let $\sigma$ denote the Frobenius element of $\Gal (\overline{k}/k)$ and also of $\Aut_F (L)$.

 For a dominant cocharacter $\mu\in X_* (T)_{\dom}$ and $b \in G(L)$ the affine Deligne-Lusztig variety is the locally closed subset
 \[
  X_\mu (b)(\overline{k}) = \{ g\cdot K \in \Grass (\overline{k}); g^{-1} b \sigma(g) \in K \mu(t) K\}.
 \]
 We equip $X_\mu (b)$ with reduced structure, making it a scheme which is locally of finite type over $\overline{k}$.  

 It is an important fact that the isomorphism class of $X_\mu(b)$ does not depend on $b$ itself but only on its $\sigma$-conjugacy class in $G(L)$. Indeed, if $b'= h^{-1}b\sigma(h)$ then $g\cdot K \rightarrow h^{-1}g \cdot K$ induces an isomorphism $X_\mu (b) \stackrel{\sim}{\rightarrow} X_\mu(b')$. 
 
 Denote by $\pi_1(G)$ the fundamental group of $G$, i.e.\ the quotient of $X_*(T)$ by the coroot lattice. By a result of Kottwitz (\cite{kottwitz85}), a $\sigma$-conjugacy class inside $G(L)$ is uniquely given by two invariants $\nu \in X_*(S)_{\QQ,\dom}$ and $\kappa \in \pi_1 (G)_I$, which are called the Newton point and the Kottwitz point of the $\sigma$-conjugacy class. We will also speak of the Newton point and the Kottwitz point of an element $b \in G(L)$ meaning the invariant associated to the $\sigma$-conjugacy class of $b$.

 We denote by $J_b$ the algebraic group whose $R$-valued points for any $F$-algebra $R$ are given by
 \[
  J_b(R) = \{g\in G(R \otimes_F L); g^{-1}b\sigma(g) = b \},
 \]
 which is an inner form of the centralizer of the Newton point of $b$ in $G_F$ (\cite{kottwitz85}, \S 5.2). Then $J_b (F)$ acts on $X_\mu (b)$ by multiplication on the left. We define the defect of $b$ to be the integer $\defect_G (b) := \rank_F G - \rank_F J_b$.

 The aim of this paper is the following theorem:

 \begin{theorem} \label{thm main}
  Assume that $X_\mu(b)$ is nonempty. Let $\nu \in X_* (S)_{\QQ,\dom}$ be the Newton point of $b$. Then
  \[
   \dim X_\mu (b) = \langle \rho, \mu - \nu \rangle - \frac{1}{2}\defect_G (b),
  \]
  where $\rho$ denotes the half-sum of all (absolute) roots of $G$.
 \end{theorem}

 It is known that $X_\mu (b)$ is non-empty if and only if the Mazur inequality holds.  Many authors have worked on this non-emptiness and have proven it in different generality, the result for unramified groups was proven by Kottwitz and Gashi (\cite{kottwitz03}, \S 4.3; \cite{gashi10} Thm.~5.2).

 The assertion of Theorem \ref{thm main} is already known in the case where $G$ is split. It is proven in the papers of G\"ortz, Haines, Kottwitz and Reuman \cite{GHKR06} and Viehmann \cite{viehmann06}. In \cite{GHKR06} the assertion is reduced to the case where $G = \GL_h$ and $b$ is superbasic, i.e.\ no $\sigma$-conjugate of $b$ is contained in a proper Levi subgroup of $G$. This case is considered in \cite{viehmann06}, where the dimension is calculated.

 The proof of Theorem \ref{thm main} is a generalization of the proof in the split case. In section \ref{sect reduction to superbasic} we reduce the theorem to the case where $G = \Res_{k'/k} GL_h$ and $b$ is superbasic. The reduction step is almost literally the same as in \cite{GHKR06}, we give an outline of the proof and explain how one has to modify the proof of \cite{GHKR06}. The rest of the paper then focuses on proving the theorem in this special case. For this we generalize the proof of Viehmann in \cite{viehmann06}. We decompose the affine Deligne-Lusztig variety using combinatorial invariants called extended EL-charts, which generalize the notion of extended semi-modules considered in \cite{viehmann06} for $G = GL_h$, and calculate the dimension of each part by generalising the computations in the $GL_h$-case. As another application of this decomposition we study the $J_b(F)$-action on the irreducible components of $X_\mu(b)$ in the superbasic case and give a conjecture on the number of orbits in the case where $\mu$ is minuscule. 

 \textbf{Acknowledgements:} I am grateful to my advisor Eva Viehmann for introducing me to this subject and for her interest and advice. Further I thank Timo Richarz for his explanations about the geometry of the affine Grassmannian and Miaofen Chen for many discussions which helped me in understanding affine Deligne-Lusztig varieties. The author was partially supported by ERC starting grant 277889 ``Moduli spaces of local $G$-shtukas''.

 \section{Reduction to the superbasic case} \label{sect reduction to superbasic}
 
 The aim of this section is to prove the following assertion.

 \begin{theorem} \label{thm reduction to superbasic}
  Assume Theorem \ref{thm main} is true for each affine Deligne-Lusztig variety $X_\mu(b)$ with $G \cong \Res_{k'/k} \GL_h$ and $b \in G(L)$ superbasic. Then it is true in general.
 \end{theorem}

 As mentioned in the introduction, we follow the proof given in \cite{GHKR06} for split groups. First we have to fix some more notation. Let

 \begin{tabular}{rl}
  $P = MN$ & be a parabolic subgroup of $G$ containing $B$. We denote by $M$ the  correspon- \\&ding Levi subgroup containing $T$ and by $N$ the unipotent radical of $P$. \\
  $\Grass,\Grass_P,\Grass_M$ & denote the affine Grassmannians of $G,P$ and  $M$  respectively. \\
  $\Grass^\omega, \Grass_M^\omega $ & denote the geometric connected component of $\Grass$ resp.\ $\Grass_M$ corresponding \\& to  $\omega \in \pi_1 (G)$ resp.\ $\omega \in \pi_1(M)$. (cf.\ \cite{PR08} Thm.\ 0.1) \\
  $x_\lambda$ & denote the image of $\lambda(t)$ in $\Grass(\overline{k})$ for $\lambda \in X_*(T)$. We use $x_0$ as ``base point'' \\& of $\Grass(\overline{k})$ . For $g \in G(L)$ we write $gx_0$ for the translate of $x_0$  w.r.t.\  the obvious \\& $G(L)$-action on $\Grass(\overline{k})$.  \\
  $X_* (T)_{\dom}$ & be the subset of $X_*(T)$ of cocharacters which are dominant w.r.t. \\& $T\subset B \subset G$.  \\
  $X_* (T)_{M-\dom}$ &  be the subset of $X_*(T)$ of cocharacters which are dominant w.r.t. \\& $T\subset B\cap M \subset M$. \\
  $R_N$ & denote the set of roots of $T$ in $\Lie N$. \\
  $\rho$ & denote the half-sum of all positive roots in $G$.  \\
  $\rho_N$ & denote the half-sum of all elements of $R_N$. \\
  $\rho_M = \rho - \rho_N$ & denote the half-sum of all positive roots in $M$.
 \end{tabular}

 Moreover, we define two partial orders on $X_*(T)$. For two cocharacters $\mu,\mu'$ we write $\mu \leq \mu'$ resp.\ $\mu \leq_M \mu'$ if $\mu'-\mu$ is a non-negative integral linear combination of simple positive coroots of $G$ resp.\ $M$.
 
 We consider $\Grass_M(\overline{k})$ as subset of $\Grass(\overline{k})$ via the obvious embedding. Furthermore the canonical morphisms $P \hookrightarrow G$ and $P \rightarrow M$ induce morphisms of ind-schemes
 
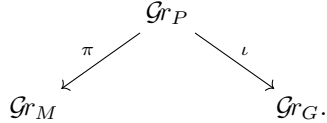
\begin{center}
  \begin{tikzcd}
   {} & \Grass_P \arrow{ld}[swap]{\pi} \arrow{rd}{\iota}
   & \\
   \Grass_M
   &
   & \Grass_G.
  \end{tikzcd}
 \end{center}
 The idea of the proof for Theorem \ref{thm reduction to superbasic} is to consider the image of an affine Deligne-Lusztig variety $X_\mu (b)$ in $\Grass_M$ under the above correspondence, assuming that $b \in M(L)$. We want to show that the image is a union of affine Deligne-Lusztig varieties, which we will later assume to be superbasic and relate the dimension of $X_\mu(b)$ to the dimension of its image.

 Let us study the diagram more thoroughly. Certainly $\pi$ is surjective and $\iota$ is bijective on geometric points by the Iwasawa decomposition of $G$. Now Lemma \ref{lem embedding grassmannian} implies that $\iota$ identifies $\Grass_P$ with a coproduct of locally closed subsets of $\Grass$, which are disjoint and cover $\Grass$. In particular we see that $X^{P\subset G}_\mu(b) := \iota^{-1}(X_\mu(b))$ is also locally of finite type and has the same dimension as $X_\mu(b)$.

 \begin{lemma} \label{lem embedding grassmannian}
  Let $i: I \hookrightarrow H$ be a closed embedding of connected algebraic groups. Then the induced map on the identity components of the affine Grassmannians $i_\calG:\Grass_I^0 \rightarrow \Grass_H^0$ is an immersion.
 \end{lemma}
 \begin{proof}
  First recall the following result in the proof of Thm.~4.5.1 of \cite{BD} (see also \cite{goertz10}, Lemma~2.12): In the case where $H/I$ is quasi-affine (resp.\ affine), the induced morphism $\Grass_I \rightarrow \Grass_H$ is an immersion (resp. closed immersion). So we want to replace $I$ by a suitable closed subgroup $I'$ which is small enough such that $H/I'$ is quasi-affine, yet big enough such that the immersion $\Grass_{I'}^0 \hookrightarrow \Grass_I^0$ is surjective.

  Now let
  \begin{center}
   \begin{tikzcd}
    0 \arrow{r}
    & R(I)_u \arrow{r}
    & I \arrow{r}
    & I_1 \arrow{r}
    & 0    
   \end{tikzcd}
  \end{center}
  be the decomposition of $I$ into a unipotent and a reductive group. We denote by $I_1^{\rm der}$ the derived group of $I_1$ and by $R(I_1)$ its radical. As $I_1/I_1^{\rm der}$ is affine, the canonical morphism $\Grass_{I_1^{\rm der}}^0 \rightarrow \Grass_{I_1}^0$ is a closed immersion. Using that $I_1 = R(I_1)\cdot I_1^{\rm der}$, we see that it is also surjective.

  We denote $I' := I \times_{I_1} I_1^{\rm der}$. As $\pi_1(I) = \pi_1(I_1)$, the canonical morphism $\Grass_{I'}^0 \rightarrow \Grass_I^0$ is the pullback of $\Grass_{I_1^{\rm der}}^0 \hookrightarrow \Grass_I^0$ and hence also a surjective immersion. Furthermore $I'$ has no non-trivial homomorphisms to $\GG_m$, hence the quotient $H/I'$ is quasi-affine and $\Grass_{I'}^0 \rightarrow \Grass_H$ is an immersion. Altogether we have
  \begin{center}
   \begin{tikzcd}[column sep = 7.5em]
    \Grass_{I'}^0 \arrow[hook]{r}[swap]{\textnormal{surj. immersion}} \arrow[bend left]{rr}{\textnormal{immersion}}
    & \Grass_I^0 \arrow{r}[swap]{i_\calG \textnormal{ monomorphism}}
    & \Grass_H^0,
   \end{tikzcd}
  \end{center}
  which proves that $i_\calG$ is an immersion. 
 \end{proof}

 In order to the determine the dimension of $X^{P\subset G}_\mu(b)$, we want to calculate the dimension of its fibres under $\pi$ and its image. For this we need a few auxiliary results. We note that the reasoning below still works if we replace $\overline{k}$ by a bigger algebraically closed field.

 We fix a dominant, regular, $\sigma$-stable coweight $\lambda_0 \in X_* (T)$. We denote for $m \in \ZZ$
 \[
  N(m) := \lambda_0(t)^m N(\overline{k}\pot{t}) \lambda_0(t)^{-m}
 \]
 Then we have a chain of inclusions $ \ldots \supset N(-1) \supset N(0) \supset N(1) \supset \ldots$ and moreover $N(L) = \bigcup_{i \in \ZZ} N(i)$. Furthermore, we note that $N(-m)/N(n)$ has a canonical structure of a variety for $m,n > 0$.

 \begin{definition}
  \begin{enumerate}
   \item A subset $Y$ of $N(L)$ is called \emph{admissible} if there exist $m,n > 0$ such that $Y \subset N(-m)$ and it is the preimage of a locally closed subset of $N(-m)/N(n)$ under the canonical projection $N(-m) \twoheadrightarrow N(-m)/N(n)$. For admissible $Y \subset N(L)$ we define the dimension of $Y$ by
   \[
    \dim Y = \dim Y/N(n) - \dim N(0)/N(n).
   \]
   \item A subset $Y$ of $N(L)$ is called \emph{ind-admissible} if $Y \cap N(-m)$ is admissible for every $m>0$. For any ind-admissible $Y \subset N(L)$ we define
   \[
    \dim Y = \sup \dim (Y \cap N(-m)).
   \]
  \end{enumerate}
 \end{definition}

 \begin{lemma} \label{lem GHKR1}
  Let $m \in M(L)$ and $\nu \in X_*(S)_{M-\dom,\QQ}$ be its Newton point. We denote $f_m: N(L) \rightarrow N(L), n \mapsto n^{-1}m\sigma(n)m^{-1}$. Then for any admissible subset $Y$ of $N(L)$ the preimage $f_m^{-1}Y$ is ind-admissible and
  \[
   \dim f_m^{-1}Y - \dim Y = \langle \rho, \nu-\nu_{\dom} \rangle.
  \]
  Moreover, $f_m$ is surjective.
 \end{lemma}
 \begin{proof}
  This assertion is the analogue of Prop.\ 5.3.1 in \cite{GHKR06}. Note that $R_N$ is $\sigma$-stable and thus the sets $N[i]$ defined in the proof of Prop.\ 5.3.2 in \cite{GHKR06} are $\sigma$-stable.
 \end{proof}

 We denote by $p_M: X_*(T) \twoheadrightarrow \pi_1(M)$ the canonical projection.

 \begin{definition}
  \begin{enumerate}
   \item For $\mu \in X_* (T)_{\dom}$ let
   \[
    S_M(\mu) := \{ \mu_M \in X_*(T)_{M-\dom};\, N(L) x_{\mu_M} \cap Kx_{\mu} \not= \emptyset\}.
   \]
   \item For $\mu \in X_* (T)_{\dom}$, $\kappa \in \pi_1(M)_I$ let
   \[
    S_M(\mu,\kappa) := \{ \mu_M \in S_M(\mu);\, \textnormal{the image of } p_M(\mu_M) \textnormal{ in } \pi_1(M)_I \textnormal{ is } \kappa\}
   \]
   \item For $\mu \in X_* (T)_{\dom}$ let
   \begin{eqnarray*}
    \Sigma(\mu) &:=& \{\mu' \in X_*(T);\, \mu'_{\dom} \leq \mu\} \\
    \Sigma(\mu)_{M-\dom} &:=& \Sigma(\mu) \cap X_*(T)_{M-\dom}
   \end{eqnarray*}
   We denote by $\Sigma(\mu)_{M-\mmax}$ the set of maximal elements in $\Sigma(\mu)_{M-\dom}$ w.r.t.\ the order $\leq_M$.
  \end{enumerate}
 \end{definition}

 \begin{lemma}
  For any $\mu \in X_*(T)_{\dom}$ we have inclusions
  \[
   \Sigma(\mu)_{M-\mmax} \subset S_M(\mu) \subset \Sigma(\mu)_{M-\dom}.
  \]
  Moreover these sets have the same image in $\pi_1(M)_I$. In particular, $S_M(\mu,\kappa)$ is nonempty if and only if $\kappa$ lies in the image of $\Sigma(\mu)_{M-\dom}$.
 \end{lemma}
 \begin{proof}
  This is (a slightly weaker version of) Lemma 5.4.1 of \cite{GHKR06} applied to $G_{k'}$. 
 \end{proof}

 \begin{definition}
  Let $\mu\in X_*(T)_{\dom}$ and $\mu_M \in \Sigma(\mu)$. We write
  \[
   d(\mu,\mu_M) := \dim (N(L)x_{\mu_M} \cap Kx_\mu).
  \]
 \end{definition}
 
 We can extend the definition above to arbitrary elements of $\Grass_M(\overline{k})$. Multiplication by an element $k_M \in K_M$ induces an isomorphism $N(L)x_{\mu_M} \cap Kx_\mu \stackrel{\sim}{\longrightarrow} N(L) k_M x_{\mu_M} \cap K x_{\mu}$, thus we have for each $m \in K_M \mu_M(t) K_M$
 \[
  \dim (N(L)mx_0 \cap Kx_\mu) = d(\mu,\mu_M).
 \]

 \begin{lemma} \label{lem GHKR2}
  Let $\mu \in X_* (T)_{\dom}$. Then for all $\mu_M \in S_M(\mu)$ we have
  \[
   d(\mu,\mu_M) \leq \langle \rho, \mu+\mu_M \rangle - 2\langle \rho_M, \mu_M\rangle
  \]
  If $\mu_M \in \Sigma(\mu)_{M-\mmax}$ this is an equality.
 \end{lemma}
 \begin{proof}
  This is Cor.~5.4.4 of \cite{GHKR06} applied to $G_{k'}$.
 \end{proof}

 For $b \in M(L)$, $\mu_M \in X_* (T)_{M-\dom}$ we denote by $X_{\mu_M}^M (b)$ the corresponding affine Deligne-Lusztig variety in the affine Grassmannian of $M$. On the contrary $X_{\mu_M} (b)$ still denotes the affine Deligne-Lusztig variety in $\Grass$, assuming that $\mu_M \in X_*(T)_{\dom}$.
 
 \begin{proposition} \label{prop GHKR3}
  Let $b \in M(L)$ be basic, i.e.\ its Newton point is central in $M$. We denote by $\kappa \in \pi_1 (M)_I$ its Kottwitz point and by $\nu \in X_*(S)_{\QQ,M-\dom}$ its Newton point.
  \begin{enumerate}
   \item The image of $X^{P\subset G}_\mu (b)$ under $\pi$ is contained in
   \[
    \bigcup_{\mu_M\in S_M(\mu,\kappa)} X_{\mu_M}^M (b).
   \]
   Denote by $\beta: X_\mu^{P\subset G}(b) \rightarrow \bigcup_{\mu_M\in S_M(\mu,\kappa)} X_{\mu_M}^M (b)$ the restriction of $\pi$.
   \item For $\mu_M \in S_M(\mu,\kappa)$ and every geometric point $x$ of $X_{\mu_M}^M(b)$ the set $\beta^{-1}(x)$ is nonempty and ind-admissible. We have
   \[
    \dim \beta^{-1} (x) = d(\mu,\mu_{M} ) + \langle \rho, \nu - \nu_{\dom} \rangle - \langle 2\rho_{N} ,\nu\rangle. 
   \]
   \item For all $\mu_M \in S_M(\mu,\kappa)$ the set $\beta^{-1} (X_{\mu_M}^M(b))$ is locally closed in $X_\mu^{P \subset G} (b)$ and
   \[
    \dim \beta^{-1} (X_{\mu_M}^M (b)) = \dim X_{\mu_M}^M (b) + d(\mu,\mu_M) + \langle \rho, \nu- \nu_{\dom}\rangle - \langle 2\rho_N, \nu \rangle.
   \]
   \item If $X_\mu (b)$ is nonempty it has dimension
   \[
    \sup\{\dim X_{\mu_M}^M (b) + d(\mu,\mu_M);\, \mu_M \in S_M(\mu,\kappa)\} + \langle \rho, \nu - \nu_{\dom} \rangle - \langle 2 \rho_N, \nu \rangle.
   \]
  \end{enumerate}
 \end{proposition}
 \begin{proof}
  This is the analogue of \cite{GHKR06}, Prop.\ 5.6.1. The proof of  (1)-(3) is the same as in \cite{GHKR06}; as this is the centerpiece of this section, we give a sketch of the proof for the readers convenience. Let $x = gx_0 \in X_\mu (b)$. We write $g = mn$ with $m \in M(L)$, $n \in N(L)$. Then
  \begin{equation} \label{eq1}
   n^{-1}m^{-1}b\sigma(m)\sigma(n) = g^{-1} b \sigma(g) \in K\mu(t)K
  \end{equation}
  As $N(L) \subset P(L)$ is a normal subgroup, this implies
  \[
   N(L) \cdot (m^{-1} b \sigma(m)) \cap K\mu(t)K \not= \emptyset.
  \]
  Thus $m^{-1}b\sigma(m) \in K\mu_M(t)K$ for a unique $\mu_M \in S_M(\mu)$, i.e.\ $\beta(x) \in X_{\mu_M}^M(b)$ proving (1).
  
  Now let $x = mx_0 \in X_{\mu_M}^M (b)$ and $b' = m^{-1} b \sigma(m)$. Then $\beta^{-1} (x)$ is the set of all $mnx_0$ satisfying (\ref{eq1}), which is equivalent to
  \[
   (n^{-1}b'\sigma(n)b'^{-1}) b' \in K\mu(t)K.
  \]
  Thus
  \[
   \beta^{-1} (x) \cong f_{b'}^{-1} (K\mu(t)Kb'^{-1} \cap N(L))/N(0).
  \]
  Hence we get
  \begin{eqnarray*}
   \dim \beta^{-1}(x) &\stackrel{{\rm Lem.}\ \ref{lem GHKR1}}{=}& \dim(K\mu(t)Kb'^{-1} \cap N(L)) - \langle \rho, \nu - \nu_{\dom} \rangle \\
   &=& (N(L)b'x_0 \cap Kx_\mu) + \dim (b'N(0)b'^{-1}) - \langle \rho, \nu-\nu_{\dom} \rangle \\
   &=& d(\mu, \mu_M) - \langle 2\rho_N, \nu \rangle + \langle \rho, \nu-\nu_{\dom} \rangle,
  \end{eqnarray*}
  where the second equality is true because $N(L)b'x_0 \cap Kx_\mu \cong (K\mu(t)Kb'^{-1} \cap N(L)) / b'N(0)b'^{-1}$. This gives (2). Now (3) follows from (2) because source and target of $\beta$ are locally of finite type over $\overline{k}$.

  Finally we prove (4). Since
  \[
   X_\mu^{P \subset G} (b) = \bigcup_{\mu_M \in S_M(\mu,\kappa)} \beta^{-1}(X_{\mu_M}^M (b))
  \]
  is a decomposition into locally closed subsets, we have
  \[
   \dim X_\mu(b) = \dim X_\mu^{P \subset G} (b) = \sup \{ \dim X_{\mu_M}^M(b);\, \mu_M \in S_M(\mu,\kappa)\}.
  \]
  Applying (3) to this formula finishes the proof.
 \end{proof}
 
 Now the main part of Theorem \ref{thm reduction to superbasic} follows:
 
 \begin{proposition}
  Let $b\in M(L)$ be basic. Assume that Theorem \ref{thm main} is true for $X_{\mu_M}^M(b)$ for every $\mu_M \in S_M(\mu,\kappa)$. Then it is also true for $X_\mu (b)$.
 \end{proposition}
 \begin{proof}
  This is a consequence of Lemma \ref{lem GHKR2} and Proposition \ref{prop GHKR3}. Its proof is literally the same as the proof of its analogue Prop.\ 5.8.1 in \cite{GHKR06}.
 \end{proof}
 
 Replacing $b$ by a $\sigma$-conjugate if necessary, we may choose a Levi subgroup $M$ such that $b$ is superbasic in $M$. As any superbasic $\sigma$-conjugacy class is basic (\cite{kottwitz85},~Prop.~6.2), the above proposition reduces Theorem \ref{thm main} to the case where $b$ is superbasic. Now it is only left to show that we may assume $G = \Res_{k'/k} \GL_h$.
 
 For this we show that it suffices to prove Theorem \ref{thm main} for the adjoint group $G^{\ad}$. We denote by subscript ``ad'' the image of elements of $G(L)$ resp.\ $X_* (T)$ resp. $\pi_1(G)$ in $G^{\ad}(L)$ resp.\ $X_*(T^{\ad})$ resp. $\pi_1 (G^{\ad})$. For $\omega \in \pi_1(G)$ we write $X_\mu(b)^\omega := X_\mu(b) \cap \Grass^\omega$. If $X_\mu(b)^\omega$ is non-empty the canonical morphism
 \begin{equation} \label{eq2}
  X_\mu(b)^\omega \to X_{\mu_{\ad}} (b_{\ad})^{\omega_{\ad}}.
 \end{equation}
 is an isomorphism. Indeed, for any reductive group $H$ over $k$ the universal covering $\tilde{H} \to H$ induces an isomorphism $\Grass_{\tilde{H}} \stackrel{\sim}{\to} \Grass^0_{H,{\rm red}}$ (\cite{PR08}~Prop.~6.1 and (6.7)), thus $\Grass_{\rm red}^0 \cong \Grass_{\tilde{G}} \cong \Grass_{G^{\ad}, {\rm red}}^0$ and by homogeneity, $\Grass_{\rm red}^\omega  \cong \Grass_{G^{\ad}, {\rm red}}^{\omega_\ad}$. Now one easily checks that $X_{\mu_{\ad}} (b_{\ad})^{\omega_{\ad}}$ is the image of $X_\mu(b)^\omega$.
 
 Now in Lemma 2.1.2 of \cite{CKV} it is proven that if $G$ is of adjoint type and contains a superbasic element $b \in G(L)$, then
 \[
  G \cong \prod_{i=1}^r \Res_{k_i/k} \PGL_{h_i},
 \]
 where the $k_i$ are finite field extensions of $k$. As
 \[
  X_{(\mu_i)_{i=1}^r} ((b_i)_{i=1}^r) \cong \prod_{i=1}^r X_{\mu_i}(b_i)
 \]
 it suffices to prove Theorem \ref{thm main} for $b$ superbasic and $G \cong \Res_{k'/k} \PGL_h$. Using the isomorphism $(\ref{eq2})$ again, we may also assume $G \cong \Res_{k'/k} \GL_h$, which finishes the proof of Theorem \ref{thm reduction to superbasic}.
 
 \section{The superbasic case: Notation and conventions} \label{sect notation}
 
 Let us first fix some basic notation. Let $X$ be a set and $v \in X^n$ with $n$ some positive integer. We then write $v_i$ for the $i$-th component of $v$. Moreover, if $v,w \in X^n$ and $X \subset \RR$ we write $v \leq w$ if $v_i \leq w_i$ for all $i$. For any real number $a$ let $\{a\} := a - \lfloor a \rfloor$ be its fractional part. We denote by $\NN$ the set of positive integers and by $\NN_0$ the set of non-negative integers.

 Let $d := [k':k] = [E:F]$, then $I \cong \ZZ / d\cdot \ZZ$. We choose the isomorphism such that $\sigma$ is mapped to $1$. From now on we only consider the case $G = \Res_{k'/k} \GL_h$ with $S \subset T \subset B \subset G$ where  $S$ and $T$ are the maximal split resp.\ maximal torus which are diagonal and $B$ is the Borel subgroup of lower triangular matrices in $G$.

 We fix a superbasic element $b\in G(L)$ with Newton point $\nu \in X_*(S)_{\QQ,\dom}$ and a cocharacter $\mu \in X_*(T)_{\dom}$. We have to show that if $X_\mu (b)$ is nonempty, we have
 \begin{align}
  \dim X_\mu (b) = \langle \rho, \mu - \nu \rangle - \frac{1}{2} \defect_G(b) \label{term superbasic}.
 \end{align}
 
 As $T$ splits over $k'$, the action of the absolute Galois group on $X_*(T)$ factorizes over $I$. We identify $X_*(T) = \prod_{\tau\in I} \ZZ^h$ with $I$ acting by cyclically permuting the factors. This yields an identification of $X_*(S) = X_*(T)^I$ with $\ZZ^h$ such that
 \[
  X_*(S) \hookrightarrow X_*(T), \nu' \mapsto (\nu')_{\tau\in I}.
 \]
 Furthermore, we denote for an element $\mu'\in X_*(T)$ by $\orbitsum{\mu'} \in  X_*(S)$ the sum of all $I$-translates of $\mu'$.  We impose the same notation as above for $X_*(T)_\QQ = \prod_{\tau\in I} \QQ^h$ and $X_*(S)_\QQ = \QQ^h$.

 We note that an element $\nu'\in \QQ^h$ is dominant if $\nu'_1 \leq \nu'_2 \leq \ldots \leq \nu'_h$ and $\mu'\in \prod_{\tau\in I} \QQ^h$ is dominant if $\mu'_\tau$ is dominant for every $\tau \in I$.

 The Bruhat order is defined on $X_*(S)_{\dom}$ resp.\ $X_*(T)_{\dom}$ such that an element $\mu''$ dominates $\mu'$ if and only if $\mu''-\mu'$ is a non-negative linear combination of relative resp. absolute positive coroots. We write $\mu' \preceq \mu''$ in this case. This motivates the following definition. For $\nu',\nu'' \in \QQ^h$ we write $\nu' \preceq \nu''$ if
 \begin{eqnarray*}
  \sum_{i=1}^j \nu'_i &\geq& \sum_{i=1}^j \nu''_i \quad \textnormal{ for all } j<n \\
  \sum_{i=1}^n \nu'_i &=& \sum_{i=1}^n \nu''_i.
 \end{eqnarray*}
 For $\mu', \mu'' \in \prod_{\tau \in I} \QQ^h$ we write $\mu' \preceq \mu''$ if $\mu'_\tau \preceq \mu''_\tau$ for every $\tau \in I$. If $\nu'$ and $\nu''$ resp.\ $\mu'$ and $\mu''$ are both dominant, this order coincides with the Bruhat order.

 For every $k$-algebra $R$ the $R$-valued points of $G$ are given by $G(R) \cong \Aut_{k' \otimes_k R} (k' \otimes_k R^h)$. We denote $N = k' \otimes_k L^h$, which is canonically isomorphic to the direct sum $\bigoplus_{\tau\in I} N_\tau$ of isomorphic copies of $L^h$. The Frobenius element $\sigma$ acts via the Galois action of $\Gal(L/F)$ on $N$, for all $\tau \in I$ we have $\sigma:N_\tau \stackrel{\sim}{\longrightarrow} N_{\tau+1}$. We fix a basis $(e_{\tau,i})_{i=1}^h$ of the $N_\tau$ such that $\varsigma(e_{\tau,i})= e_{\varsigma\tau,i}$ for all $\varsigma \in I$. For $\tau \in I, l \in \ZZ, i=1,\ldots ,h$ denote $e_{\tau,i+l\cdot h} := t^l\cdot e_{\tau,i}$. Then each $v\in N_\tau$ can be written uniquely as infinite sum 
 \[
  v = \sum_{n \gg -\infty} a_n\cdot e_{\tau,n} \\ 
 \]
 with $a_n \in k$.

 Now we denote by $M^0$ the $\overline{k}\pot{t}$-submodule of $N$ generated by the $e_{\tau,i}$ for $i\geq 0$. With respect to our choice of basis, $K$ is the stabilizer of $M^0$ in $G(L)$ and $g\mapsto gM^0$ defines a bijection
 \[
  \Grass(\overline{k}) \cong \{M= \prod_{\tau\in I} M_\tau;\, M_\tau \textnormal{ is a lattice in } N_\tau\}.
 \]

 Suppose we are given two lattices $M,M' \subset L^h$. By the elementary divisor theorem we find a basis $v_1, \ldots, v_n$ of $M$ and a unique tuple of integers $a_1 \leq \ldots \leq a_n$ such that $t^{a_1}v_1,\ldots , t^{a_n}v_n$ form a basis of $M'$. We define the cocharacter $\inv (M,M'):\GG_m \rightarrow \GL_h, x \mapsto \diag (x^{a_1}, \ldots , x^{a_n})$. If we write $M' = gM$ with $g \in GL_h (L)$ we may equivalently define $\inv (M,M')$ to be the unique cocharacter of the diagonal torus which is dominant w.r.t.\ the Borel subgroup of lower triangular matrices and satisfies $g \in \GL_h(\overline{k}\pot{t}) \inv(M,M')(t) \GL_h(\overline{k}\pot{t})$.

 In terms of the notation introduced above we have
 \[
  X_\mu(b)(\overline{k}) \cong \{ (M_\tau \subset N_\tau \textnormal{ lattice})_{\tau\in I};\, \inv (M_\tau, b\sigma (M_{\tau-1})) = \mu_\tau\}.
 \]
 
 \begin{definition}
  \begin{enumerate}
   \item We call a tuple of lattices $(M_\tau \subset N_\tau)_{\tau\in I}$ a \emph{$G$-lattice}.
   \item We define the \emph{volume} of a $G$-lattice $M = gM^0$ to be the tuple
   \[
    \vol (M) = (\val \det g_\tau)_{\tau \in I}.
   \]
    Similarly, we define the volume of $M_\tau$ to be $\val\det g_\tau$. We call $M$ \emph{special} if $\vol (M) = (0)_{\tau \in I}$.
  \end{enumerate}
 \end{definition}

 The assertion that $b$ is superbasic is by \cite{CKV} equivalent to $\nu$ being of the form $(\frac{m}{d\cdot h}, \frac{m}{d \cdot h}, \ldots , \frac{m}{d \cdot h})$  with $(m,h) = 1$. Then by \cite{kottwitz03}, Lemma 4.4 $X_\mu (b)$ is nonempty if and only if $\nu$ and $\mu$ have the same image in $\pi_1(G)_I$, which is equivalent to $\sum_{\tau \in I, i=1,\ldots h} \mu_{\tau,i} = m$. We assume that this equality holds from now on.

 Furthermore we have for each central cocharacter $\nu' \in X_*(S)$ the obvious isomorphism 
 \[
  X_\mu (b) \stackrel{\sim}{\rightarrow} X_{\mu+\nu'} (\nu'(t)\cdot b).
 \]
 So we may (and will) assume that $\mu \geq 0$, which amounts to saying that we have $b\sigma (M) \subset M$ for $G$-lattices $M \in X_\mu (b) (\overline{k})$.

 Since the affine Deligne-Lusztig varieties of two $\sigma$-conjugated elements are isomorphic, we can assume that $b$ is the form $b(e_{\tau,i}) = e_{\tau,i+m_\tau}$ where $m_\tau = \sum_{i=1}^h \mu_{\tau,i}$. We could have chosen any tuple of integers $(m_\tau)$ such that $\sum_{\tau\in I} m_\tau = m$ but this particular choice has the advantage that the components of any $G$-lattice in $X_\mu(b)$ have the same volume. In general,
 \begin{eqnarray*}
  \vol M_\tau - \vol M_{\tau-1} &=& (\vol M_\tau - \vol b\sigma(M_{\tau-1})) + (\vol b\sigma(M_{\tau-1}) - \vol M_{\tau-1}) \\
  &=& (\sum_{i=1}^h \mu_{\tau,i}) - m_\tau.
 \end{eqnarray*}

 Recall that the geometric connected components of $\Grass$ are in bijection with $\pi_1(G) = \ZZ^I$. This bijection is given by mapping a $G$-lattice to its volume. Thus the subsets of lattices $\Grass$ resp.\ $X_\mu (b)$ obtained by restricting the value of the volume of the components are open and closed. Denote by $X_\mu(b)^i \subset X_\mu(b)$ the subset of all $G$-lattices $M$ such that $M_0$ (or equivalently every $M_\tau$) has volume $i$. Let $\pi \in J_b(F)$ be the element with $\pi(e_{\tau,i}) = e_{\tau,i+1}$ for all $\tau\in I, i\in\ZZ$. Then $g\cdot K \mapsto \pi g \cdot K$ defines an isomorphism $X_\mu(b)^i \stackrel{\sim}{\longrightarrow} X_\mu(b)^{i+1}$. Thus $\dim X_\mu (b) = \dim X_\mu (b)^0$ so that it is enough to consider the subset of special lattices.

 \section{Polygons}

 In this section we introduce our notion of polygons and reformulate the formula (\ref{term superbasic}) in terms of this notion. For this we need to introduce some more notation.

 We denote by $(\QQ^h)^0$ resp.\ $(\prod_{\tau \in I} \QQ^h)^0$ the subspaces of $\QQ^h$ resp.\ $\prod_{\tau \in I} \QQ^h$ generated by the relative resp.\ absolute coroots. Explicitly, these subspaces are given by
 \begin{eqnarray*}
  (\QQ^h)^0 &=& \{ \nu' \in \QQ^h;\, \nu'_1 + \ldots + \nu'_h = 0\} \\
  (\prod_{\tau\in I} \QQ^h)^0 &=& \{ \mu' \in \prod_{\tau\in I} \QQ^h;\, \mu'_\tau \in (\QQ^h)^0 \textnormal{ for every } \tau \in I\}.
 \end{eqnarray*}
 We fix lifts $\omega_i$ and $\omega_{i,\tau}$ of relative resp.\ absolute fundamental weights of the derived group to $X^*(S)$ resp.\ $X^*(T)$. We thus have for $\nu' \in (\QQ^h)^0$ and $\mu' \in (\prod_{\tau\in I} \QQ^h)^0$
 \begin{eqnarray*}
  \langle \omega_i, \nu' \rangle &=& -\sum_{j=1}^i \nu'_j \\
  \langle \omega_{\tau,i}, \mu' \rangle &=& -\sum_{j=1}^i \mu'_{\tau,j}.
 \end{eqnarray*}

 \begin{definition} 
  \begin{enumerate}
   \item For $\nu' \in \QQ^h$ let
    \[ 
     [ \nu' ] := \sum_{i=1}^{h-1} \lfloor \langle \nu', \omega_i \rangle \rfloor.
    \]
   \item For $\nu',\nu'' \in \QQ^h$ we define
   \[
    \length{\nu'}{\nu''} := [-\nu' ] + [\nu''].
   \]
   \item For $\mu',\mu'' \in \prod_{\tau \in I} \QQ^h$ let 
   \[
    \lengthG{\mu'}{\mu''} := \length{\orbitsum{\mu'}}{\orbitsum{\mu''}}.
   \]
  \end{enumerate}
 \end{definition}
 
 Now we give a geometric interpretation of $\length{\nu'}{\nu''}$ in terms of polygons in a special case that covers all applications in this paper. We start with the following observation.

 \begin{lemma} \label{lem trivial facts}
  Let $\nu', \nu'' \in \QQ^h$ with $\nu' \preceq \nu''$ and $\nu'' \in \ZZ^h$. Then
  \begin{enumerate}
   \item $\length{\nu'}{\nu''} = [\nu'' - \nu']$.
   \item $\length{\nu'}{\nu''}$ is independent of the choice of lifts of the fundamental weights.
  \end{enumerate}
 \end{lemma}

 \begin{proof}
  (1) This is an easy consequence of the fact that $\langle \omega_i,\nu'', \rangle$ is an integer for all $i$.

  (2) As $\nu''-\nu' \in (\QQ^h)^0$, the value $[\nu''-\nu']$ is independent of our choice of lifts. Together with part (1) this proves the claim.
 \end{proof}

 \begin{definition}
  To an element $\nu' \in\QQ^h$ we associate a polygon $\calP (\nu')$ which is defined over $[0,h]$ with starting point $(0,0)$ and slope $\nu'_i$ over $(i-1,i)$. We also denote by $\calP (\nu')$ the corresponding piecewise linear function on $[0,h]$.
 \end{definition}

 Let $\nu'$ and $\nu''$ be as in Lemma \ref{lem trivial facts}. Now $\nu' \preceq \nu''$ amounts to saying that $\calP(\nu')$ is above $\calP(\nu'')$ and that these two polygons have the same endpoint. It follows from the first assertion of the lemma that $\length{\nu'}{\nu''}$ is equal to the number of lattice points which are on or below $\calP(\nu')$ and above $\calP(\nu'')$.

 \begin{figure}[h]
  \begin{center}
   \begin{tikzpicture}
    \draw[step=1cm,gray, ultra thin]  (0,0) grid (7.4,3.4);

    \draw[->,very thick,darkgray] (0,0) -- (0,3.5);
    \draw[->,very thick,darkgray] (0,0) -- (7.5,0);

    \draw[thick] (0,0) -- (7,3);
    \draw[thick] (0,0) -- (5,0) --(6,1) -- (7,3);

    \fill (3,1) circle (2pt);
    \fill (4,1) circle (2pt);
    \fill (5,1) circle (2pt);
    \fill (5,2) circle (2pt);
    \fill (6,2) circle (2pt);
   \end{tikzpicture}

   \caption{Geometric interpretation of $\length{\nu'}{\nu''} = 5$ for $\nu' = (\frac{3}{7}, \frac{3}{7}, \frac{3}{7}, \frac{3}{7}, \frac{3}{7}, \frac{3}{7}, \frac{3}{7}), \nu'' = (0,0,0,0,0,1,2)$.}
  \end{center}
 \end{figure}
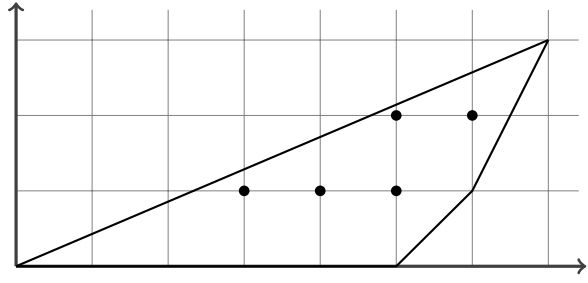

 \begin{proposition} \label{prop superbasic}
  We have
  \[
   \lengthG{\mu}{\nu} = \langle\rho,\mu-\nu\rangle - \frac{1}{2}. \defect_G(b)
  \]
  In particular the formula (\ref{term superbasic}) is equivalent to
  \[
   \dim X_\mu(b) = \lengthG{\mu}{\nu}.
  \]
 \end{proposition}
 
 \begin{remark}
  This formula coincides with the formula conjectured by Rapoport in \cite{rapoport06}, p. 296 up to a minor correction. Rapoport's formula becomes correct if one multiplies all cocharacters by $d$.
 \end{remark}

 In order to prove this proposition, we need the following lemmas.

 First we need the following fact from Bruhat-Tits theory, which holds in greater generality than just our specific situation. For a reductive group $H$ over a quasi-local field we denote by $\BT(H)$ its Bruhat-Tits building.

 \begin{lemma} \label{lem BT}
  Assume that $F$ is a quasi-local field and $L$ the completion of its maximal unramified extension. Let $H$ be a reductive group over $F$ and $A$ a maximal split torus of $H_L$ defined over $F$. If its apartment $\fraka$ contains a $\Gal(F^{nr}/F)$-stable alcove $C$, then $A$ contains a maximal split torus of $H$. 
 \end{lemma}
 \begin{proof}
  We identify $\BT(H)$ with the $\Gal(F^{nr}/F)$-fixed points in $\BT(H_L)$. Then $\BT(H) \cap C$ is a nonempty open subset of $\BT(H)$. In particular we have
  \[
   \dim \fraka^{\Gal(F^{nr}/F)} \geq \dim C^{\Gal(F^{nr}/F)} = \dim \BT(H) = \rank H.
  \]
  Thus $A$ contains a maximal split torus of $H$.
 \end{proof}

 \begin{lemma}
  \[
   \frac{1}{2}\defect_G (b) = \sum_{i=1}^{h-1} \left\{ \langle \omega_i, \orbitsum{\nu} \rangle \right\}
  \]
 \end{lemma}
 \begin{proof}
  The analogous assertion for split reductive groups with simply connected derived group was proven by Kottwitz in \cite{Kottwitz06}. We modify his proof in order to get the result in the case of $\Res_{k'/k} \GL_h$. 

  We consider the groups $T' \subset B' \subset GL_h$ where $T'$ is the diagonal torus and $B'$ the Borel subgroup of \emph{upper} triangular matrices. Denote by $W$ resp.\ $\widetilde{W}$ the Weyl group resp.\ extended affine Weyl group of $\GL_h$. We identify $W$ with the symmetric group $S_h$ and $\widetilde{W}$ with $\ZZ^h \rtimes S_h$. The canonical projection of the stabilizer $\Omega$ of the standard (upper triangular) Iwahori subgroup in $\widetilde{W}$ to $\pi_1 (G)$ is an isomorphism. Thus we get an embedding
  \[ 
   \ZZ \cong \pi_1 (\GL_h) \cong \Omega \hookrightarrow \widetilde{W}, n \mapsto \widetilde{w}_n := \left( (1,0,\ldots,0)\cdot \left( \begin{array}{cccc} 1 & 2 & \cdots & h \\ 2 & 3 & \cdots & 1 \end{array} \right) \right)^n.
  \]
  We denote by $w_n$ the image of $\widetilde{w}_n$ w.r.t.\ the canonical projection $\widetilde{W} \twoheadrightarrow W$.
 
  Decompose $b = (b_\tau)_{\tau\in I}$ according to $\Res_{k'/k} \GL_h (L) \cong \prod_{\tau \in I} \GL_h (L)$. Then $b_\tau$ is the generalized permutation matrix representing $\tilde{w}_{m_\tau}$ in $\GL_h (L)$. Denote by $\fraka = \prod_{\tau\in I} \fraka_\tau \cong \prod_{\tau\in I} \RR^h$ the apartment of $G_L$ corresponding to $T$. Now $\BT(J_b)$ is canonically isomorphic to the fixed points of the Bruhat-Tits building $\BT(G_L)$ of $G_L$ of $\sigma' := \Int (b) \circ \sigma = (w_{m_\tau})_{\tau \in I} \cdot \sigma$. Since the standard Iwahori  is $\sigma'$-stable, we get by Lemma \ref{lem BT}
  \begin{eqnarray*}
   \rank J_b &=& \dim \fraka^{(w_{m_\tau})_{\tau\in I} \cdot \sigma} \\
    &=& \dim\{(v_\tau) \in \fraka; w_{\tau+1} (v_\tau) = v_{\tau+1}\} \\
    &=& \dim\{v_0 \in \fraka_0; w_{m_0}\cdot w_{m_{d-1}} \cdot \ldots \cdot w_{m_1} (v_0) = v_0\} \\
    &=& \dim\{v_0 \in \fraka_0; w_m (v_0) = v_0\}.
  \end{eqnarray*}
  Now we can reduce to the case $d=1$: Denote by $G_0 \cong \GL_h$ the factor of $G_L$ corresponding to $\tau = 0$ with diagonal torus $T_0$ and lower triangular Borel subgroup $B_0$. We identify the root data of $G_0$ with the relative root data of $G$. Now we apply the longest Weyl group element $w$ to our formula to compensate the change of Borel subgroups and then apply \cite{Kottwitz06}, Theorem 1.9.2 to finish the proof.
  \[
   \dim \fraka_0 - \dim \fraka_0^{w_m} = \dim \fraka_0 - \dim\fraka_0^{w\cdot w_m \cdot w} = \sum_{i=1}^{h-1} \left\{ \langle \omega_i, (\frac{m}{h}, \ldots , \frac{m}{h}) \right\} = \sum_{i=1}^{h-1} \left\{ \langle \omega_i, \orbitsum{\nu} \rangle \right\}
  \]
 \end{proof}

 \begin{proof}[Proof of Proposition \ref{prop superbasic}]
  Using the lemma above, we get that
  \begin{eqnarray*}
   \langle \rho, \mu - \nu \rangle -\frac{1}{2}\cdot\defect_G (b) &=& \sum_{i=1,\ldots,h-1 \atop \tau\in I} \langle \omega_{\tau,i}, \mu - \nu \rangle - \sum_{i=1}^{h-1} \left\{ \langle \omega_i, \orbitsum{\nu}\rangle \right\} \\
   &=& \sum_{i=1}^{h-1} \langle \sum_{\tau\in I} \omega_{\tau,i}, \mu - \nu \rangle - \sum_{i=1}^{h-1} \left\{\langle \omega_{h-i}, \orbitsum{\nu} \rangle \right\} \\
   &=& \sum_{i=1}^{h-1} (\langle \omega_i, \orbitsum{\mu}\rangle + \langle \omega_i, -\orbitsum{\nu} \rangle) - \sum_{i=1}^{h-1} \left\{\langle \omega_i, -\orbitsum{\nu} \rangle \right\} \\
   &=& \lengthG{\nu}{\mu}
  \end{eqnarray*}
 \end{proof}

 Finally we prove two lemmas which we will use in section \ref{sect combinatorics}. The reader may skip the rest of this section for the moment.

 \begin{lemma} \label{lem length to dom}
  Let $\nu'\in\ZZ^h$. Then 
  \[
   \length{\nu'}{\nu'_{\dom}} = \sum_{1\leq i < j \leq h} \max\{ \nu'_i - \nu'_j, 0\} .
  \]
 \end{lemma}
 \begin{proof}
  The assertion follows from the following observation. If $\nu'' \in \ZZ^h$ with $\nu''_i > \nu''_{i+1}$ and we swap these coordinates, then $[\nu'']$ is reduced by the difference of these two values. Now $\nu'_{\dom}$ is obtained from $\nu'$ by carrying out the above transposition repeatedly until the coordinates are in increasing order. Since we have swapped the coordinates $\nu'_i$ and $\nu'_j$ during this construction if and only if $i<j$ and $\nu'_i > \nu'_j$ we get the above formula.
 \end{proof}

 \begin{lemma} \label{lem length of transfer}
  Let $\nu' \in \ZZ^{h}$ be dominant, $ 1\leq i \leq j \leq h, \beta \in \ZZ_{\geq 0}$ and 
  \[
   \nu'' := (\nu'_1,\ldots,\nu'_{i-1},\nu'_i-\beta,\nu'_{i+1},\ldots,\nu'_{j-1},\nu'_j + \beta, \nu'_{j+1}, \ldots, \nu'_h).
  \]
  Then
  \[
   \length{\nu'}{\nu''_{\dom}} = \left( \sum_{k=1}^\beta \sum_{l=\nu'_i - \beta}^{\nu'_j -1 } |\{n;\, \nu'_n = k+l \}| \right) - \beta.
  \]
 \end{lemma}
 \begin{proof}
  Obviously we have
  \[
   \length{\nu'}{\nu''} = (j-i)\cdot \beta = \left( \sum_{i \leq n \leq j} \beta \right) - \beta
  \]
  and by the previous lemma
  \[
   \length{\nu''}{\nu''_{\dom}} = \sum_{n < i:\, \nu'_i-\beta< \nu'_n}  (\nu'_n-(\nu'_i-\beta)) + \sum_{n>j:\, \nu'_n < \nu'_j + \beta} (\nu'_j + \beta - \nu'_n).
  \]
  Using $\length{\nu'}{\nu''_{\dom}} = \length{\nu'}{\nu''} + \length{\nu''}{\nu''_{\dom}}$ one easily deduces the above assertion.
 \end{proof}

 \begin{corollary} \label{cor length of transfer}
  Let $\nu',\nu'' \in \ZZ^d$ be dominant with $\nu' \preceq \nu''$ such that the multiset of their coordinates differs by only two elements. Say $n_2,n_3$ in the multiset of coordinates of $\nu'$ are replaced by $n_1,n_4$ in the multiset of coordinates of $\nu''$ with $n_1 \leq n_2 \leq n_3 \leq n_4$. Then
  \[
   \length{\nu'}{\nu''} = \left( \sum_{k=0}^{n_4-n_3-1} \sum_{l=0}^{n_4-n_2-1} |\{n;\, \nu'_n = n_4-k-l-1\}| \right) +n_1 - n_2 .
  \]
 \end{corollary}
 \begin{proof}
  The assertion is just a reformulation of the previous lemma.
 \end{proof}

 \section{Extended EL-charts} \label{sect EL-charts}
 In order to calculate the dimension of the affine Deligne-Lusztig variety, we decompose $X_\mu^0 (b)$ as follows. Denote
 \begin{eqnarray*}
  \calI_\tau: N_\tau \setminus \{ 0\} \quad &\rightarrow& \ZZ \\
  \sum_{n \gg -\infty} a_n\cdot e_{\tau,n} &\mapsto& \min \{n\in\ZZ; a_n \not= 0\}.
 \end{eqnarray*} 
 Note that $\calI_\tau$ satisfies the strong triangle inequality for every $\tau$. We denote $N_{hom} := \coprod_{\tau \in I} (N_\tau\setminus \{0\})$, analogously $M_{hom}$, and define the index map
 \[
  \calI := \sqcup\, \calI_\tau : N_{hom} \rightarrow  \coprod_{\tau\in I} \ZZ
 \]
  For $M \in X_\mu (b)^0 (\overline{k})$, we define
 \[
  A(M) := \calI (M_{hom})
 \]
 and a map $\varphi(M):  \coprod_{\tau\in I} \ZZ \rightarrow \ZZ\cup\{-\infty\}$ such that for $a \in A(M)$
 \[
  \varphi(M)(a) = \max\{n\in \NN_0;\, \exists v \in M_{hom} \text{ with } \calI (v) = a, t^{-n}b\sigma (v) \in M_{hom} \} 
 \]
 and $\varphi(M)(a) = - \infty$ otherwise.
 Now we decompose $X_\mu (b)^0$ such that $(A(M), \varphi(M))$ is constant on each component. We will discuss the properties  of this decomposition in section \ref{sect decomposition}. In this section we give a description of the invariants $(A,\varphi)$.

 \begin{definition}
  Let $\ZZ^{(d)} := \coprod_{\tau \in I} \ZZ_{(\tau)}$ be the disjoint union of $d$ isomorphic copies of $\ZZ$. For $a\in\ZZ$ we denote by $a_{(\tau)}$ the corresponding element of $\ZZ_{(\tau)}$ and write $|a_{(\tau)}| := a$. We equip $\ZZ^{(d)}$ with a partial order ``$\leq$'' defined by
  \[
   a_{(\tau)} \leq c_{(\varsigma)} :\Lra a\leq c \textnormal{ and } \tau = \varsigma
  \]
  and a $\ZZ$-action given by
  \[
   a_{(\tau)} + n = (a+n)_{(\tau)}.
  \]
  Furthermore we define a function $f:\ZZ^{(d)} \rightarrow \ZZ^{(d)}, a_{(\tau)} \mapsto (a+m_{\tau+1})_{(\tau+1)}$.
 \end{definition}

  We impose the notation that for any subset $A \subset \ZZ^{(d)}$ we write $A_{(\tau)} := A\cap  \ZZ_{(\tau)}$.
 
 \begin{definition}
  \begin{enumerate}
   \item Let $d,h$ be positive integers and $(m_\tau)_{\tau \in I} \in \ZZ^d$ such that $m := \sum_{\tau\in I} m_\tau$ and $h$ are coprime, let $f$ be defined as above. An \emph{EL-chart} for $(\ZZ^{(d)},f,h)$ is a nonempty subset $A \subset \ZZ^{(d)}$ which is bounded from below, stable under $f$ and satisfies $A+h \subset A$.
   \item Let $A$ be an EL-chart and $B = A\setminus (A+h)$. We say that $A$ is \emph{normalized} if $\sum_{b_{(0)} \in B_{(0)}} b = \frac{h\cdot (h-1)}{2}$.
  \end{enumerate}
 \end{definition}

 Our next aim is to give a characterization of EL-charts. For this let $A$ be an EL-chart and $B := A\setminus(A+h)$. Obviously $|B| = d\cdot h$. We define a sequence $b_0, \ldots b_{d\cdot h -1}$ of distinct elements of $B$ as follows. Denote by $b_0$ the minimal element of $B_{(0)}$. If $b_i$ is already defined, we denote by $b_{i+1}$ the unique element which can be written as
 \[
  b_{i+1} = f(b_i) - \mu'_{i+1}\cdot h
 \]
 for some $\mu'_{i+1} \in \ZZ$. These elements are indeed distinct: If $b_i = b_j$ then obviously $i \equiv j \mod d$ and then $b_{i + k\cdot d} \equiv b_i + k\cdot m \mod h$ implies that $i=j$ as $m$ and $h$ are coprime. This reasoning also shows that if we define $b_{d\cdot h}$ according to the recursion formula above, we get $b_{d\cdot h} = b_0$. Therefore we will consider the index set of the $b_i$ and $\mu'_i$ as $\ZZ/dh\ZZ$.
 We define
 \[
  \suc (b_i) := b_{i+1}
 \]
 and call $\mu' = (\mu'_i)_{i=1,\ldots, d\cdot h}$ the type of $A$.

 At some point, it may be helpful to distinguish the $b_i$'s and $\mu'_i$'s of different components. For this we may change the index set to $I \times \{1, \ldots ,h\}$ via
 \begin{eqnarray*}
  b_{\tau,i} &:=& b_{\tau + (i-1)d} \\
  \mu'_{\tau,i} &:=& \left\{ \begin{array}{ll} \mu'_{\tau + (i-1)d} & \textnormal{ if } \tau \not= 0 \\ \mu'_{id} & \textnormal{ if } \tau=0 \end{array}\right. .
 \end{eqnarray*}
 Here we choose that standard set of representatives $\{0,\ldots,d-1\} \subset \ZZ$ for $I$.

 With the change of notation we have that $b_{\tau,i} \in B_{(\tau)}$ for all $i,\tau$ and that $b_{0,1}$ is the minimal element of $B_{(0)}$ and we have the recursion formula
 \begin{eqnarray*}
  b_{\tau+1,i} &=& f(b_{\tau,i}) - \mu'_{\tau+1,i}h \quad\quad\textnormal{ if } \tau \not= d-1 \\
  b_{0,i} &=& f(b_{d-1,i-1}) - \mu'_{0, i-1}h.
 \end{eqnarray*}

 \begin{lemma}
  \begin{enumerate}
   \item For every EL-chart $A$ there exists a unique integer $n$ such that $A+n$ is normalized.
   \item Mapping an EL-chart to its type induces a bijection between normalized EL-charts and the set $\{\mu' \in\prod_{\tau\in I} \ZZ^h; \orbitsum{\mu}' \succeq \orbitsum{\nu}\}$.
  \end{enumerate}
 \end{lemma}
 
 \begin{proof}
  (1) In order to obtain a normalized EL-chart, we have to choose $n = \frac{1}{h} \cdot \left(\frac{h\cdot (h-1)}{2} - \sum_{b_{(0)} \in B_{(0)}} b\right)$. Since by definition every residue modulo $h$ occurs exactly once in $B_{(0)}$, this is indeed an integer. \\ \smallskip
  (2) Since an EL-chart $A$ is uniquely determined by $A\setminus (A+h)$ which is, up to $\ZZ$-action, uniquely determined by the type of $A$, we know the type induces an injection on the set of normalized EL-charts into $\prod_{\tau\in I} \ZZ^h$. The condition $ b_0 = \min\{b_{k\cdot d} \mid 1\leq k \leq h-1\}$ translates to the condition $\orbitsum{\mu}' \succeq \orbitsum{\nu}$ on the type, which is shown by the following computation. For $1\leq k\leq h-1$ we have
  \begin{equivarray}
   & b_0 \leq b_{kd} \\
   \Lra & b_0 \leq b_0 + k\cdot m - \sum_{i=1}^{k} (\orbitsum{\mu}'_i) \cdot h \\
   \Lra & \sum_{i=1}^k \orbitsum{\mu}'_i \leq k\cdot \frac{m}{h} \\
   \Lra & \sum_{i=1}^k \orbitsum{\mu}'_i  \leq \sum_{i=1}^k \orbitsum{\nu}_i
  \end{equivarray}
  Similarly one shows the equivalence of $b_0 = b_{hd}$ and $\sum_{i=1}^h \orbitsum{\mu}'_i  = \sum_{i=1}^h \orbitsum{\nu}_i$. Thus if $\mu'$ is the type of an EL-chart, we have $\orbitsum{\mu}' \succeq \orbitsum{\nu}$. On the other hand this also shows that any such $\mu'$ is the type of some EL-chart and thus by (1) also the type of a normalized EL-chart.
 \end{proof}

 \begin{definition}
  For $a\in A$ we call $\height (a) := \max \{n\in\NN_0; a-n\cdot h \in A\}$ the \emph{height of $a$}.
 \end{definition}

 \begin{definition} \label{def extended EL-chart}
  \begin{enumerate}
   \item An \emph{extended EL-chart} is a pair $(A,\varphi)$ where $A$ is a normalized EL-chart and $\varphi: \ZZ^{(d)} \rightarrow \NN_0\cup\{-\infty\}$ such that the following conditions hold for every $a\in\ZZ^{(d)}$.
   \begin{enumerate}
    \item $\varphi(a) = -\infty$ if and only if $a\not\in A$.
    \item $\varphi(a+h) \geq \varphi (a) +1$.
    \item $\varphi(a) \leq \height f(a)$ if $a \in A$ with equality if $\{c\in\ZZ^{(d)};\, c \geq a \} \subset A$.
    \item $ |\{c \in \ZZ^{(d)};\, c \geq a \} \cap \varphi^{-1} ( \{n\} )| \leq |\{c\in\ZZ^{(d)};\, c \geq a+h \} \cap \varphi^{-1} (\{ n+1\} )|$ for all $n\in \NN_0$.
   \end{enumerate}
   \item An extended EL-chart is called \emph{cyclic} if equality holds in (c) for every $a\in A$.
   \item The \emph{Hodge-point} of an extended EL-chart $(A,\varphi)$ is the dominant cocharacter $\mu'' \in \prod_{\tau\in I} \ZZ^h$ for which the coordinate $n$ occurs with multiplicity $|A_{(\tau)} \cap \varphi^{-1} (\{ n \})| - |A_{(\tau)} \cap \varphi^{-1} (\{ n-1 \})|$ in $\mu''_\tau$. We also say that $(A,\varphi)$ is an extended EL-chart for $\mu''$.
  \end{enumerate}
 \end{definition}

 \begin{remark}
 We point out that as an obvious consequnce of condition (d) we have for any integer $n$
 \[
 |A_{(\tau)} \cap \varphi^{-1} (\{ n \})| - |A_{(\tau)} \cap \varphi^{-1} (\{ n-1 \})| \geq 0,
 \]
 thus the construction of the Hodge point described above is feasible.
  Because of condition (c) we have $|A_{(\tau)} \cap \varphi^{-1}( \{n\} )| = h$ for every $\tau\in I$ and sufficiently large $n$. Hence the Hodge point is indeed an element of $\prod_{\tau\in I} \ZZ^h$.
 \end{remark}

 Except for condition (d) the definition of an EL-chart is obviously a generalization of Definition 3.4 in \cite{viehmann06}. As we will frequently refer to Viehmann's paper we give an equivalent condition for (d) which is easily seen to be a generalization of condition (4) in her definition. However, we will not use this assertion in the sequel.

 \begin{lemma} \label{lem adapted family}
  For every $(A,\varphi)$ satisfying conditions (a)-(c) of Definition \ref{def extended EL-chart}, (d) may equivalently be replaced with the following condition. For every $\tau$, we can write $A_{(\tau)} = \bigcup_{l=1}^h \{ a^{\tau,l}_j \}_{j=0}^\infty$ with
   \begin{enumerate}
    \item[(a)] $\varphi (a^{\tau,l}_{j+1}) = \varphi (a^{\tau,l}_j) + 1$
    \item[(b)] If $\varphi(a^{\tau,l}_j + h) = \varphi(a^{\tau,l}_j)+1$, then $a^{\tau,l}_{j+1} = a^{\tau,l}_j+h$, otherwise $a^{\tau,l}_{j+1} > a^{\tau,l}_j+h$.
   \end{enumerate}
  Then the Hodge point is the dominant cocharacter associated to $(a^{\tau,l}_0)_{l=1,\ldots, h \atop \tau \in I}$.
 \end{lemma}
 \begin{proof}
  If we have a decomposition of $A$ as above, it is obvious that (d) is true. Now let $(A,\varphi)$ be an extended EL-chart. We construct the sequences $(a^{\tau,l}_j)_{j\in\NN}$ separately for each $\tau$. So fix $\tau\in I$. We construct the sequences by induction on the value of $\varphi$. Take every element of $A$ for which $\varphi$ has minimal value as initial element for some sequence. Now if we have sorted all elements $a\in A$ with $\varphi (a) \leq n$ in sequences $(a^{\tau,l}_j)_{j\in\NN}$ we proceed as follows. Condition (d) of Definition \ref{def extended EL-chart} guarantees that we can continue all our sequences such that they satisfy (a) and (b). If there are still some $a\in A$ with $\varphi (a) = n+1$ which are not already an element of a sequence, we take them as initial objects for some sequences. Since $|\varphi^{-1} (n) \cap A_{(\tau)}| = h$ for $n \gg 0$, we get indeed $h$ sequences.
 \end{proof}

 \begin{lemma} \label{lem cyclic EL-chart}
  Let $A$ be an EL-chart of type $\mu'$. There exists a unique $\varphi_0$ such that $(A,\varphi_0)$ is a cyclic extended EL-chart. The Hodge point of $(A,\varphi_0)$ is $\mu'_{\dom}$.
 \end{lemma}
 \begin{proof}
  The function $\varphi_0: \ZZ^{(d)} \rightarrow \NN_0 \cup \{-\infty\}$ is uniquely determined by equality in (c) and condition (a). For any $a\in A$ we get $\varphi_0 (a+h) = \varphi_0 (a) +1$, which proves (b) and (d). The second assertion follows from $\mu'_{i+1} = \varphi_0 (b_i)$.
 \end{proof}
 
 The following construction will help us to deduce assertions for general extended EL-charts from the assertion in the cyclic case.

 \begin{definition} \label{def deformation}
  Let $(A,\varphi)$ be an extended EL-chart and $(A,\varphi_0)$ the cyclic extended EL-chart associated to $A$. For any $\tau \in I$ we denote
  \[
   \{x_{\tau,1},\ldots ,x_{\tau,n_\tau}\} = \{a\in A_{(\tau)};\, \varphi (a+h) > \varphi(a) + 1 \}
  \]
  where the $x_{\tau,i}$ are arranged in decreasing order. We write $\mathbf{n} := (n_\tau)_{\tau\in I}$. For $0 \leq \mathbf{i} \leq \mathbf{n}$ (i.e. $0 \leq i_\tau \leq n_\tau$ for all $\tau$) let
  \[
   \varphi_{\mathbf{i}} = \left\{ \begin{array}{ll} -\infty & \textnormal{if } a \not\in A, \\ \varphi_0 (a) & \textnormal{if } a\in A_{(\tau)} \textnormal{ and } i_\tau = 0 \\ \varphi(a) & \textnormal{if } a \in A_{(\tau)}, i_\tau > 0 \textnormal{ and } a \geq x_{\tau, i_\tau}, \\  \varphi(a+h)-1 & \textnormal{otherwise}. \end{array} \right.
  \]
  We call the family $(A,\varphi_{\mathbf{i}})_{0 \leq \mathbf{i} \leq \mathbf{n}}$ the \emph{canonical deformation} of $(A,\varphi)$.
 \end{definition}

 One easily checks that the $(A,\varphi_{\mathbf{i}})$ are indeed extended EL-charts (the properties (a)-(d) of Definition \ref{def extended EL-chart} follow from the analogous properties of $(A,\varphi)$) and that $\varphi_{\mathbf{i}} = \varphi_0$ for $\mathbf{i} = (0)_{\tau\in I}$ and $\varphi_{\mathbf{i}} = \varphi$ for $\mathbf{i} = \mathbf{n}$. Denote the Hodge-point of $(A,\varphi_{\mathbf{i}})$ by $\mu^{\mathbf{i}}$.

 We note that one can define the $\varphi_{\mathbf{i}}$ recursively. Let $\varsigma \in I$ and $0 \leq \mathbf{i} \leq \mathbf{i'} \leq \mathbf{n}$ with $i'_{\varsigma} = i_{\varsigma}+1$ and $i'_{\tau} = i_\tau$ for $\tau \not= \varsigma$. We denote $\alpha := \varphi(x_{\varsigma, i_\varsigma}+h) - (\varphi(x_{\varsigma, i_\varsigma}) + 1)$. Then
 \[
  \varphi_{\mathbf{i'}} (a) = \left\{ \begin{array}{ll}
                                         \varphi_{\mathbf{i}} (a) - \alpha & \textnormal{if } a = x_{\varsigma,i_\varsigma}, x_{\varsigma,i_\varsigma} - h, \ldots , x_{\varsigma,i_\varsigma} - \height (x_{\varsigma, i_\varsigma}) \cdot h \\
                                         \varphi_{\mathbf{i}} (a) & \textnormal{otherwise}
                                        \end{array} \right. 
 \]

 \begin{lemma} \label{lem comparison type Hodge}
  Let $(A,\varphi)$ be an extended EL-chart of type $\mu'$ with Hodge point $\mu$. Then $\mu'_{\dom} \preceq \mu$. Furthermore, we have $\mu'_{\dom} = \mu$ if and only if $(A,\varphi)$ is cyclic.
 \end{lemma}
 \begin{proof}
  We have already shown that $\mu = \mu'_{\dom}$ if $(A,\varphi)$ is cyclic in Lemma \ref{lem cyclic EL-chart}. It suffices to show $\mu^{\mathbf{i}} \prec \mu^{\mathbf{i'}}$ for all pairs $\mathbf{i}, \mathbf{i'}$ such that $i'_{\varsigma} = i_{\varsigma}+1$ for some $\varsigma \in I$ and $i'_{\tau} = i_\tau$ for $\tau \not= \varsigma$. From the recursive description of $\varphi_{\mathbf{i'}}$ above we see that we get $\mu^{\mathbf{i'}}$ from $\mu^{\mathbf{i}}$ by replacing two coordinates in $\mu^{\mathbf{i}}_\varsigma$ and permuting its coordinates if necessary to get a dominant cocharacter. Using the same notation as above, we replace the set
  \[
   \{ \varphi_{\mathbf{i}}(x_{\varsigma,i_\varsigma})-\height (x_{\varsigma, i_\varsigma}), \varphi_{\mathbf{i}}(x_{\varsigma,i_\varsigma}) - \alpha +1\}
  \]
  with
  \[
    \{ \varphi_{\mathbf{i}} (x_{\varsigma,i_\varsigma})-\alpha-\height (x_{\varsigma, i_\varsigma}), \varphi_{\mathbf{i}} (x_{\varsigma,i_\varsigma})+1  \}.
  \]
  Since
  \[
   \varphi_{\mathbf{i}}(x_{\varsigma,i_\varsigma})-\height (x_{\varsigma, i_\varsigma}),\, \varphi_{\mathbf{i}}(x_{\varsigma,i_\varsigma}) - \alpha +1 \in (\varphi_{\mathbf{i}} (x_{\varsigma,i_\varsigma})-\alpha-\height (x_{\varsigma, i_\varsigma}), \varphi_{\mathbf{i}'} (x_{\varsigma,i_\varsigma})+1),
  \]
  we get $\mu^{\mathbf{i}} \prec \mu^{\mathbf{i'}}$.
 \end{proof}

 Deducing the following corollaries from Lemma \ref{lem comparison type Hodge} is literally the same as the proofs of Cor.\ 3.7 and Lemma 3.8 in \cite{viehmann06}. We give the proofs for the reader's convenience.

 \begin{corollary}
  If $\mu$ is minuscule, then all extended EL-charts for $\mu$ are cyclic.
 \end{corollary}
 \begin{proof}
  Let $(A,\varphi)$ be an extended EL-chart for $\mu$ and let $\mu'$ be the type of $A$. Since $\mu$ is minuscule, $\mu'_{\dom} \preceq \mu$ implies $\mu'_{\dom} = \mu$. Hence the assertion follows from Lemma \ref{lem comparison type Hodge}.
 \end{proof}

 \begin{corollary} \label{cor finitely many EL-charts}
  There are only finitely many extended EL-charts for $\mu$.
 \end{corollary}
 \begin{proof}
  As a consequence of Lemma \ref{lem comparison type Hodge} there are only finitely many possible types of extended EL-charts with Hodge point $\mu$. If we fix such a type, the EL-chart $A$ is uniquely determined. The value of the function $\varphi$ is uniquely determined by $A$ for all but finitely many elements  and for each such element, $\varphi$ can only take finitely many values by the inequality of Definition \teil{\ref{def extended EL-chart}}{c}.
 \end{proof}

 \section{Decomposition of $X_\mu (b)$} \label{sect decomposition}

 We fix a lattice $M \in X_\mu (b)^0$. Let $(A(M),\varphi(M))$ be defined as above and $B(M) := A(M) \setminus (A(M)+h)$.
 
 \begin{lemma} \label{lem nak}
  Let $a_{(\tau)} \in A(M)$ such that $c\in A(M)$ for every $c \geq a_{(\tau)}$. Then $\{v \in N_\tau; \calI_\tau (v) \geq a\} \subset M_{\tau}$. \end{lemma}
 \begin{proof}
  We denote $M' := \{v\in N_{\tau}; \, \calI_\tau(v) \geq a\}$ and $M'' := M' \cap M_{\tau}$. For $b=a,\ldots, a+h-1$ choose $v_b \in M''$ with $\calI_\tau (v_b) = b$. Obviously we can write any element $x\in M'$ in the form
  \[
    x = \sum_{b=a}^{a+h-1} \alpha_b\cdot v_b + x' 
  \]
  with $\alpha_b \in k$ and $x' \in t\cdot M' = \{ v\in N_\tau;\, \calI(v) \geq a+h\}$. Thus $M' = M'' + t\cdot M'$ and the claim follows by Nakayama's lemma.
 \end{proof}

 \begin{lemma} \label{lem A(M) is extended EL-chart}
  Let $M \in X_\mu(b)^0$. Then $(A(M), \varphi(M))$ is an extended EL-chart for $\mu$.
 \end{lemma}
 
 \begin{proof}
  Let us first check that $A(M)$ is a normalized EL-chart. It is stable under $f$ and the addition of $h$ since
  \begin{eqnarray}
   \calI (t\cdot v) &=& \calI(v) + h \\
   \calI (b\sigma (v)) &=& f( \calI(v) )
  \end{eqnarray}
  and $t\cdot M \subset M$ and $b\sigma (M) \subset M$ by our convention $\mu \geq 0$. The fact that $A(M)$ is bounded from below is obvious, thus $A(M)$ in an EL-chart.
  
  The assetion that $A(M)$ is normalised can be deduced from the prerequisite $\vol M = (0)_{\tau\in I}$ as follows. Let $a \in \NN$ such that $c \in A(M)$ for every $c \geq (a\cdot h)_{(0)}$. Then by Lemma \ref{lem nak}, we have $t^a\cdot M^0_{(0)} \subset M$ and
 \begin{eqnarray*}
  0 &=& \vol(M_{(0)}) \\
  &=& \dim_k (M_{(0)}/t^a\cdot M^0_{(0)}) - \vol(t^a\cdot M^0_{(0)}) \\
  &=& |\{c \in A(M);\, c <  (a\cdot h)_{(0)}\}| - a\cdot h  \\
  &=& |\NN^d \setminus A(M)_{(0)}| - |A(M)_{(0)} \setminus \NN^d|.
  \end{eqnarray*}
  Hence
  \[
   \sum_{b \in B(M)_{(0)}} b = \sum_{i=0}^{h-1} i = \frac{h(h-1)}{2}.
  \]
  
  Now $\varphi(M)$ satisfies property (a) of Definition \ref{def extended EL-chart} by definition. To see that it satisfies (b) and (c), fix $a \in A$ and let $v\in M_{hom}$ such that $\calI (v) = a$ and $t^{-\varphi(a)} \cdot b\sigma(v) \in M$. Then $t^{-\varphi(a)+1} b \sigma(t\cdot v) = t^{-\varphi(M)(a)} \cdot b\sigma (v) \in M$, proving (b) and $f(a) - \varphi(M) (a)\cdot h = \calI(t^{-\varphi(a)}\cdot b\sigma(v)) \in A(M)$ which implies the inequality part of (c). Let $a \in A$ such that $c \in A(M)$ for every $c \geq a$. We choose an element $v' \in M_{hom}$ with $\calI(v') = f(a) - h\cdot\height f(a)$ and denote $v = t^{\height f(a)}\cdot (b\sigma)^{-1}(v')$. Then $\calI(v) = a$, hence $v \in M$ by Lemma \ref{lem nak} and $t^{-\height f(a)}\cdot b\sigma(v) = v' \in M$. Thus $\varphi(a) = \height f(a)$. To verify that $\varphi$ has property (d), we fix $\tau \in I$ and define for $a\in\ZZ_{(\tau)} \cup \{-\infty\}, n\in\NN$ the $\overline{k}$-vector space
  \[
   V'_{a,n} := \{ v \in M_\tau;\, v=0 \text{ or } \calI (v) \geq a, t^{-n}\cdot b\sigma(v) \in M \}
  \]
  and $V_{a,n} := V'_{a,n}/V'_{a,n+1}$. Now associate to every $c \in \{a'\in A_{(\tau)};\, a' \geq a\} \cap \varphi^{-1}(\{n\})$ an element $v_c \in M_\tau$ with $\calI (v_c) = c$ and $t^{-\varphi(c)}\cdot v_c \in M$. Using the strong triangle inequality for $\calI_\tau$, we see that the images $v_c$ in $V_{a,n}$ are linearly independent. Thus $\dim V_{a,n} \geq | \{a'\in A_{(\tau)};\, a' \geq a\} \cap \varphi(M)^{-1}(\{n\})|$. By counting dimensions in a suitable finite dimensional quotient of $V'_{a,0}$, we see that this is in fact an equality. Now the images of the $t\cdot v_c$ in $V_{a+h,n+1}$ are also linearly independent, thus
  \[
   |\{a'\in A_{(\tau)};\, a' \geq a\} \cap \varphi(M)^{-1}(\{n\})| \leq | \{a'\in A_{(\tau)};\, a' \geq a+h\} \cap \varphi(M)^{-1}(\{n+1\})|.
  \]
  
  Now it remains to show that $(A(M),\varphi (M))$ has Hodge point $\mu$. But
  \[
   |\{i; \mu_{\tau, i} =n \}| = \dim V_{-\infty,n} - \dim V_{-\infty, n-1} = | A_{(\tau)} \cap \varphi^{-1} (\{n\})| - |A_{(\tau)} \cap \varphi^{-1}(\{n-1\})|.
  \]
 \end{proof}

 This proof also shows that $A(M)$ is an EL-chart for every $G$-lattice $M \subset N$ and $A(M)$ is normalized if and only if $M$ is special. For any extended EL-chart $(A,\varphi)$ for $\mu$ we denote
 \[
  \calS_{A,\varphi} = \{ M \in \Grass (\overline{k});\, (A(M),\varphi(M)) = (A,\varphi)\}.
 \]
  Since the Hodge point of $M$ and $(A(M),\varphi(M))$ coincide by the lemma above, we have indeed $\calS_{A,\varphi} \subset X_\mu(b)^0$.

 \begin{lemma}
  The $\calS_{A,\varphi}$ define a decomposition of $X_\mu(b)^0$ into finitely many locally closed subsets. In particular, $\dim X_\mu(b)^0 = \max_{(A,\varphi)} \dim \calS_{A,\varphi}$.
 \end{lemma}
 \begin{proof}
  By Lemma \ref{lem A(M) is extended EL-chart}, $X_\mu(b)^0$ is the (disjoint) union of the $\calS_{A,\varphi}$ and by Corollary \ref{cor finitely many EL-charts} this union is finite. It remains to show that $\calS_{A,\varphi}$ is locally closed. One shows that the condition $(A(M)_{(\tau)},\varphi(M)_{|A(M)_{(\tau)}}) = (A_{(\tau)}, \varphi_{|A_{(\tau)}})$ is locally closed analogously to the proof of Lemma 4.2 in \cite{viehmann06}. Then it follows that $\calS_{A,\varphi}$ is locally closed as it is the intersection of finitely many locally closed subsets. 
 \end{proof}
  
  \begin{definition}
  Let $(A,\varphi)$ be an extended EL-chart for $\mu$. We define
 \[
  \calV(A,\varphi) = \{ (a,c) \in A \times A;\, a < c, \varphi(a) > \varphi(c) > \varphi(a-h)\} 
 \]
 \end{definition}

 We note that the set $\calV(A,\varphi)$ is finite, since we have $\varphi(a-h) = \varphi(a)-1$  for almost all $a\in A$.

 \begin{proposition}\label{prop decomposition}
  Let $(A,\varphi)$ be an extended EL-chart for $\mu$. There exists an open subscheme $U_{A,\varphi} \subseteq \AA^{|\calV (A,\varphi)|} $ and a morphism $U_{A,\varphi} \rightarrow \calS_{A,\varphi}$ which is bijective on $\overline{k}$-valued points. In particular, $ \calS_{A,\varphi}$ is irreducible and of dimension $\calV(A,\varphi)|$.
 \end{proposition}

 \begin{proof}
   The proof is almost the same as of Thm.\ 4.3 in \cite{viehmann06}. We give an outline of the proof and explain how to adapt the proof of Viehmann to our more general notion.

   For any $\overline{k}$-algebra $R$ and $x \in R^{\calV(A,\varphi)} = \AA^{\calV(A,\varphi)}(R)$ we denote the coordinates of $x$ by $x_{a,c}$. We associate to every $x$ a set of elements $\{ v(a) \in N_{hom}; a\in A\}$ which satisfies the following equations.
  
  If $a = y := \max\{b\in B_{(0)}\}$ then
  \[
   v(a) = e_a + \sum_{(a,c)\in\calV(A,\varphi)} x_{a,c}\cdot v(c).
  \]
  For any other element $a \in B$ we want
  \[
    v(a) = v' + \sum_{(a,c)\in\calV(A,\varphi)} x_{a,c}\cdot v(c)
  \]
  where $v' = t^{-\varphi(a')} \cdot b\sigma(v(a'))$ for $a'$ being minimal satisfying $f(a') - \varphi(a')\cdot h = a$. At last, if $a \not\in B$, we impose
  \[
    v(a) = t\cdot v(a-h) + \sum_{(a,c)\in\calV(A,\varphi)} x_{a,c}\cdot v(c).
  \]
   
  \begin{claim}
   The set $\{v(a); a\in A\}$ is uniquely determined by the equations above.
  \end{claim}
  
  Hence the rule $x \mapsto M(x) := \langle v(a) ;\, a\in A \rangle_{k\pot{t}}$ is well-defined and as it is obviously functorial, induces a morphism $\AA^{\calV(A,\varphi)} \rightarrow \Grass$. But the image of this morphism is in general not contained in $\calS_{A,\varphi}$, we only have the following assertions:
  
  \begin{claim}
   For every $x \in \AA^{\calV(A,\varphi)}(\overline{k})$ we have $A(M(x)) = A$ and $\varphi(M(x))(a) \geq \varphi(a)$ for every $a\in A$.
  \end{claim}
  
  \begin{claim}
   The preimage $U(A,\varphi)$ of $\calS_{A,\varphi}$ is nonempty and open in $\AA^{\calV(A,\varphi)}$.
  \end{claim}
  
  Now the fact that the restriction $U(A,\varphi) \rightarrow \calS_{A,\varphi}$ of above morphism defines a bijection of $\overline{k}$-valued points follows from the following assertion.
  
  \begin{claim}
   Let $M \subset N$ be a special $F$-lattice such that $(A(M),\varphi(M)) = (A,\varphi)$. Then there exists a unique set of elements $\{v(a);\, a \in A\} \subset M$ satisfying the equations above.
  \end{claim}
  
  It remains to prove the four claims. But their proofs are literally the same as in \cite{viehmann06} if one replaces ``$a+m$'' and ``$a+im$'' by ``$f(a)$'' respectively ``$f^i (a)$''.
 \end{proof}

 \section{Combinatorics for extended EL-charts} \label{sect combinatorics}
 Proposition \ref{prop decomposition} reduces the proof of the formula (\ref{term superbasic}) to an estimation of $|\calV(A,\varphi)|$ for extended EL-charts $(A,\varphi)$ with Hodge point $\mu$. We start with the case where $(A,\varphi)$ is cyclic. In this case we have $\varphi(a+h) = \varphi(a)+1$ for all $a\in A$ and thus
 \[
  \calV(A,\varphi) = \{ (b,c) \in B \times A;\, b<c, \varphi(b) > \varphi(c) \}
 \]

 \begin{proposition} \label{prop upper bound V}
  Let $(A,\varphi)$ the cyclic extended EL-chart of type $\mu$. Then \[|\calV(A,\varphi)| \geq \lengthG{\nu}{\mu}\]. 
 \end{proposition}
 \begin{proof}
  First we show that the right hand side of the inequality counts the number of positive integers $n$ such that $b_0+n\not\in A_{(0)}.$ Indeed, as $A_{(0)}+h \subset A_{(0)}$, we have
  \begin{eqnarray*}
   |\{n\in\NN; b_0+n \not\in A_{(0)}\}| &=& \frac{1}{h}\cdot\sum_{j=1}^{h-1} \left(b_{j\cdot d} - b_0 - j\right) \\
   &=& \frac{1}{h}\cdot\sum_{j=1}^{h-1} \left(f^{j\cdot d} (b_0) - \left(\sum_{i=1}^{j\cdot d} \mu_{i}\cdot h\right) - b_0 - j \right) \\
   &=& \frac{1}{h}\cdot\sum_{j=1}^{h-1} \left(b_0+j\cdot m - \left(\sum_{i=1}^{j} \orbitsum{\mu}_i \cdot h \right) -b_0-j \right) \\
   &=& \frac{1}{h}\cdot\sum_{j=1}^{h-1} \left( j\cdot (m-1) - h \cdot \calP (\orbitsum{\mu})(i) \right) \\
   &=& \frac{(m-1)(h-1)}{2} - \sum_{i=1}^{h-1}\calP (\orbitsum{\mu})(i) \\
   &=& \sum_{i=1}^{h-1} (\lfloor \calP(\orbitsum{\nu})(i) \rfloor - \calP (\orbitsum{\mu})(i)) \\
   &=& \lengthG{\nu}{\mu}
  \end{eqnarray*}
  Now we construct an injective map from $\{n\in\NN;b_0+n\not\in A\}$ into $\calV(A,\varphi)$. For this we remark that $(b_i,b_i +n) \in \calV (A,\varphi)$ if and only if $b_i +n\in A$ and  $b_{i+1}+n \not\in A$. Thus $n \mapsto (b_i,b_i+n)$ where $i = \max\{i=1,\ldots,h-1;\, b_i+n \in A \}$ gives us the injection we sought. Note this map is well-defined since for any $n\in\NN$ and maximal element $b$ of $B$ we have $b+n \in A$.
 \end{proof}

 \begin{theorem}\label{thm lower bound V}
  Let $(A,\varphi)$ be an extended EL-chart for $\mu$. Then $|\calV(A,\varphi)| \leq \lengthG{\nu}{\mu}$.
 \end{theorem}

 \begin{proof}
  We assume first that $(A,\varphi)$ is cyclic with type $\mu'$. Then
  \begin{eqnarray*}
   |\calV(A,\varphi)| &=& |\{(b_i,a) \in B \times A;\, b_i < a, \varphi(a) < \varphi(b_i)\}| \\
   &=& |\{ (b_i,b_j+\alpha h); \, \alpha \in \NN_0, b_i < b_j, \mu'_{i+1} > \mu'_{j+1}+\alpha\}| \\
   & & + |\{ (b_i,b_j+\alpha h); \, \alpha \in \NN, b_j < b_i < b_j+\alpha h, \mu'_{i+1} > \mu'_{j+1}+\alpha\}| \\
   &=& \sum_{(b_i,b_j);\, b_i < b_j, \mu'_{i+1} > \mu'_{j+1}} \mu'_{i+1} - \mu'_{j+1} \\
   & & + |\{ (b_i,b_j+\alpha h); \, \alpha \in \NN, b_j < b_i < b_j+\alpha h, \mu'_{i+1} > \mu'_{j+1}+\alpha\}|
  \end{eqnarray*}
  We refer to these two summands by $S_1$ and $S_2$.

  For each $\tau\in I$ denote by $(\tilde{b}_{\tau,1},\tilde{\mu}_{\tau+1,1}),\ldots,(\tilde{b}_{\tau,h}, \tilde{\mu}_{\tau+1,h})$ the permutation of $(b_{\tau,1},\mu_{\tau+1,1}),\ldots,$ $(b_{\tau,h},\mu_{\tau+1,h})$ such that the $(\tilde{b}_{\tau,i})_i$ are arranged in increasing order. From the ordering we obtain for any $1\leq j \leq h, \tau\in I$
  \[
   \sum_{i=0,\ldots ,j-1} \tilde{b}_{\tau,i} \leq \sum_{i=0,\ldots ,j-1} \suc (\tilde{b}_{\tau-1,i})
  \]
  and thus
  \[
   \sum_{i=1}^{j} |\tilde{b}_{\tau,i}| - |\tilde{b}_{\tau-1,i}| \leq j\cdot m_\tau - \sum_{i=1}^j \tilde{\mu}_{\tau,i}\cdot h .
  \]
  Adding these inequalities for all $\tau \in I$ and rearranging the terms we obtain $\sum_{i=1}^{j} \orbitsum{\tilde{\mu}}_j \leq j\cdot\frac{m}{h}$. Thus $\underline{\nu} \preceq \underline{\tilde{\mu}} \preceq \underline{\mu}$.

  Using this notation we can simplify
  \begin{eqnarray*}
   S_1 &=& \sum_{i < j \atop \tau\in I} \max\{\tilde{\mu}_{\tau,i} - \tilde{\mu}_{\tau,j}, 0\} \\
   &=& \sum_{\tau\in I} \length{\tilde{\mu}_\tau}{\tilde{\mu}_{\tau\, \dom}} \\
   &=& \lengthG{\tilde{\mu}}{\mu},
  \end{eqnarray*}
  where the second line holds because of Lemma \ref{lem length to dom}.

 We have now reduced the claim to $S_2 \leq \lengthG{\nu}{\mu} - \lengthG{\tilde{\mu}}{\mu}$, which is equivalent to $S_2 \leq \lengthG{\nu}{\tilde{\mu}}$. Now
  \begin{eqnarray*}
   S_2 &=& \sum_{i=2}^{h} \sum_{j=1}^{i-1} \sum_{\tau\in I} | \{\alpha \in \ZZ; \tilde{b}_{\tau,i} < \tilde{b}_{\tau,j} + \alpha h, \tilde{\mu}_{\tau+1,i} > \tilde{\mu}_{\tau+1,j} + \alpha\} | \\
   &=& \sum_{i=2}^{h} \sum_{j=1}^{i-1} \sum_{\tau\in I} | \{\alpha\in\ZZ;\, \tilde{b}_{\tau,i} - \tilde{b}_{\tau,j} < \alpha h < \tilde{\mu}_{\tau+1,i}h - \tilde{\mu}_{\tau+1,j}h\} | \\
   &=& \sum_{i=2}^{h} \sum_{j=1}^{i-1} \sum_{\tau\in I} |\{\alpha\in\ZZ;\, 0 < \alpha h < (\tilde{b}_{\tau,j} - \tilde{\mu}_{\tau+1,j} h) - (\tilde{b}_{\tau,i} - \tilde{\mu}_{\tau+1,i}h)\}| \\
   &=& \sum_{i=2}^{h} \sum_{j=1}^{i-1} \sum_{\tau\in I} \max\{ \lfloor \tfrac{ \suc (\tilde{b}_{\tau,j}) - \suc (\tilde{b}_{\tau,i})}{h} \rfloor,0\}
  \end{eqnarray*}

  Recall that $\lengthG{\nu}{\tilde{\mu}}$ counts the lattice points between the polygons associated to $\orbitsum{\nu}$ and $\orbitsum{\tilde{\mu}}$. So it is enough to construct a decreasing sequence (with respect to $\preceq$) of $\widehat{\psi}^i \in \QQ^h$ for $i=1,\ldots ,h$ with $\widehat{\psi}^1 = \orbitsum{\nu}$ and $\widehat{\psi}^{h} = \orbitsum{\tilde{\mu}}$ such that there are at least
  \[
   \sum_{j=1}^{i-1} \sum_{\tau\in I} \max\{ \lfloor \tfrac{ \suc (\tilde{b}_{\tau,j}) - \suc (\tilde{b}_{\tau,i})}{h} \rfloor,0\}
  \]
  lattice points which are on or  below $\calP(\widehat{\psi}^{i-1})$ and above $\calP(\widehat{\psi}^{i})$. 

  We define a bijection $\suc_i:B\rightarrow B$ as follows: For $j > i,\tau\in I$ let $ \suc_i (\tilde{b}_{\tau,j}) = \suc (\tilde{b}_{\tau,j})$.
  For $j\leq i$ define $\suc_i (\tilde{b}_{\tau,j})$ such that for every $\tau \in I$ the tuple  $( \suc_i (\tilde{b}_{\tau,j}) )_{j=1}^i$ is the permutation of $(\suc (\tilde{b}_{\tau, j}))_{j=1}^i$ which is arranged in increasing order. Let $\psi^i \in \prod_{\tau\in I} \QQ^h$ be defined by $\suc_i (\tilde{b}_{\tau,j}) = f(\tilde{b}_{\tau,j}) - \psi^i_{\tau+1,j}\cdot h$ and $\widehat{\psi}^i = \orbitsum{\psi^i}$. By definition we have
  \[
   \psi_{\tau+1,j}^i = \frac{m_{\tau+1}}{h} - \frac{\suc_i (\tilde{b}_{\tau,j}) - \tilde{b}_{\tau,j}}{h}
  \]
  and thus $\widehat{\psi}^1 = \orbitsum{\tilde{\mu}}$ and $\widehat{\psi}^{h} = \orbitsum{\nu}$. 

 We now estimate the number of lattice points between $\calP(\widehat{\psi}^{i-1})$ and $\calP(\widehat{\psi}^{i})$ by calculating $\calP (\widehat{\psi}^i) (j) - \calP (\widehat{\psi}^{i-1}) (j)$ for every $1\leq j < h$. In particular, we will see that $\calP (\widehat{\psi}^i) (j) - \calP (\widehat{\psi}^{i-1}) (j) \geq 0$ and thus $\widehat{\psi}^i \preceq \widehat{\psi}^{i-1}$. In order to compare $\calP (\widehat{\psi}^i)$ and $\calP (\widehat{\psi}^{i-1})$, we give an equivalent \emph{recursive} description of $\suc_i$. Let $i_0 \leq i$ be minimal such that $\suc_{i-1} (\tilde{b}_{\tau,i_0}) \geq \suc (\tilde{b}_{\tau,i})$. Then
  \[
   \suc_i (\tilde{b}_{\tau,j}) = \left\{ \begin{array}{ll}  \suc_{i-1} (\tilde{b}_{\tau,j}) & \textnormal{ if } j < i_0 \\ \suc(\tilde{b}_{\tau,i}) & \textnormal{ if } j = i_0 \\ \suc_{i-1} (\tilde{b}_{\tau,j-1}) & \textnormal{ if } i_0 < j \leq i \\ \suc_{i-1} (\tilde{b}_{\tau,j}) & \textnormal{ if } j > i \end{array} \right. .
  \]
  Now
  \begin{eqnarray*}
   \calP (\widehat{\psi}^i) (j) - \calP (\widehat{\psi}^{i-1}) (j) &=& \sum_{\tau\in I} \sum_{k=1}^j ({\psi}_{\tau,k}^i - {\psi}_{\tau,k}^{i-1}) \\
   &=& \sum_{\tau\in I} \sum_{k=1}^j \frac{1}{h}(\suc_{i-1} (\tilde{b}_{\tau,k}) - \suc_i (\tilde{b}_{\tau,k})) \\
   &=& \sum_{\tau\in I} \frac{1}{h} \left( \sum_{k=1}^j \suc_{i-1} (\tilde{b}_{\tau,k}) - \sum_{k=1}^j \suc_i (\tilde{b}_{\tau,k}) \right).
  \end{eqnarray*}
  By the recursive formula above the right hand side equals zero if $j \geq i$ and
  \[
   \sum_{\tau\in I} \max \{0, \tfrac{\suc_{i-1}(\tilde{b}_{\tau.j}) - \suc (\tilde{b}_{\tau,i})}{h}\}
  \]
  if $j < i$. Thus there are at least
  \[
   \sum_{\tau\in I}\sum_{j<i} \max \{0, \lfloor \tfrac{\suc_{i-1}(\tilde{b}_{\tau.j}) - \suc (\tilde{b}_{\tau,i})}{h} \rfloor \} =\sum_{\tau\in I}\sum_{j<i} \max \{0, \lfloor \tfrac{\suc(\tilde{b}_{\tau.j}) - \suc (\tilde{b}_{\tau,i})}{h}\rfloor \}
  \]
  lattice points which are above $\calP (\widehat{\psi}^i)$ and on or below $\calP (\widehat{\psi}^{i-1})$, which finishes the proof for a cyclic EL-chart.

  For a noncyclic EL-chart $(A,\varphi)$ consider the canonical deformation $((A,\varphi_{\mathbf{i}}))_{\mathbf{i}}$ (see Definition \ref{def deformation}). The theorem is proven by induction on $\mathbf{i}$. For $\mathbf{i} = (0)_{\tau\in I}$ the extended EL-chart is cyclic and the claim is proven above. Now the induction step requires that we show that the claim remains true if increase a single coordinate of $\mathbf{i}$ by one. Let $\mathbf{i'} \leq \mathbf{n}$ with $i'_\varsigma = i_\varsigma+1$ for some $\varsigma \in I$ and $i'_\tau = i_\tau$ for $\tau \not= \varsigma$. For convenience, we introduce the notations
  \begin{eqnarray*}
   \alpha &:=& \varphi(x_{\varsigma,i_\varsigma}+h) - (\varphi(x_{\varsigma,i_\varsigma})+1) \\
   n &:=& \height (x_{\varsigma,i_\varsigma}) \\
   \mu^{i} &:=& \mu^{\mathbf{i}}_\varsigma \\
   \mu^{i'} &:=& \mu^{\mathbf{i'}}_\varsigma.
  \end{eqnarray*}

  Then the right hand sides of the formula (\ref{term superbasic}) for $\mu^{\mathbf{i'}}$ and $\mu^{\mathbf{i}}$ differ by
  \[
   \langle \rho, \mu^{\mathbf{i'}} \rangle - \langle \rho, \mu^{\mathbf{i}} \rangle = \lengthG{\mu^{\mathbf{i'}}}{\mu^{\mathbf{i}}}.
  \] 
 We had shown in the proof of Lemma~\ref{lem comparison type Hodge} that one obtains $\mu^{\mathbf{i'}}$ from $\mu^{\mathbf{i}}$ by replacing the entries $ \varphi_{\mathbf{i}}(x_{\varsigma,i_\varsigma})-\height (x_{\varsigma, i_\varsigma})$ and $\varphi_{\mathbf{i}}(x_{\varsigma,i_\varsigma}) - \alpha +1$ in the $\varsigma$-th component by $ \varphi_{\mathbf{i}} (x_{\varsigma,i_\varsigma})-\alpha-\height (x_{\varsigma, i_\varsigma})$ and $\varphi_{\mathbf{i}} (x_{\varsigma,i_\varsigma})+1 $ and rearranging the entries afterwards so that we obtain a domiant cocharacter. Thus we may apply Corollary \ref{cor length of transfer}, which shows that
  \begin{eqnarray*}
   \lengthG{\mu^{\mathbf{i'}}}{\mu^{\mathbf{i}}} = \length{\mu^{i'}}{\mu^{i}} 
    &= &(\sum_{k=0}^{\alpha-1} \sum_{l=0}^{n} |\{j; \mu^i_j = \varphi_{\mathbf{i}}(x_{\varsigma,i_\varsigma})-k-l\}|) - \min \{\alpha,n+1\}.
  \end{eqnarray*}  
  
  We denote this term by $\Delta$. We have to show that $\Delta \geq |\calV(A,\varphi_{\mathbf{i'}})| - |\calV(A,\varphi_{\mathbf{i}})|$. 
 Now recall that 
  \[
  \varphi_{\mathbf{i'}} (a) = \left\{ \begin{array}{ll}
                                         \varphi_{\mathbf{i}} (a) - \alpha & \textnormal{if } a = x_{\varsigma,i_\varsigma}, x_{\varsigma,i_\varsigma} - h, \ldots , x_{\varsigma,i_\varsigma} - \height (x_{\varsigma, i_\varsigma}) \cdot h \\
                                         \varphi_{\mathbf{i}} (a) & \textnormal{otherwise.}
                                        \end{array} \right. 
 \]
 Together with the explanation given below one obtains
  \begin{eqnarray*}
   \calV(A,\varphi_{\mathbf{i'}}) \setminus \calV(A,\varphi_{\mathbf{i}}) &=& D_1 \cup D_3 \\
   \calV(A,\varphi_{\mathbf{i}}) \setminus \calV(A, \varphi_{\mathbf{i'}}) &=& D_2
  \end{eqnarray*}
  where
  \begin{eqnarray*}
   D_1 &=& \{ (x_{\varsigma,i_\varsigma}+h,c) \in A\times A;\, c > x_{\varsigma,i_\varsigma}+h, \varphi_{\mathbf{i'}}(x_{\varsigma,i_\varsigma}+h) > \varphi_{\mathbf{i'}} (c) > \varphi_{\mathbf{i'}} (x_{\varsigma,i_\varsigma}) \} \\   
   D_2 &=& \left\{(x_{\varsigma,i_\varsigma}-nh, c) \in A\times A;\, \begin{array}{l}
                                                                 c > x_{\varsigma,i_\varsigma}-nh, \varphi_{\mathbf{i}}(x_{\varsigma,i_\varsigma}-nh) > \varphi_{\mathbf{i}}(c), \\ \varphi_{\mathbf{i'}}(x_{\varsigma,i_\varsigma}-nh) \leq \varphi_{\mathbf{i'}}(c)
                                                                \end{array}
  \right\} \\
   D_3 &=& \left\{(b,x_{\varsigma,i_\varsigma}-\delta h) \in B\times A;\, \begin{array}{l}
                                                                      \delta \in \{0,\ldots,n\}, b \not=  x_{\varsigma,i_\varsigma}-nh, b < x_{\varsigma,i_\varsigma} - \delta h, \\ \varphi_{\mathbf{i'}}(b) > \varphi_{\mathbf{i'}} (x_{\varsigma,i_\varsigma} - \delta h), \varphi_{\mathbf{i}}(b) \leq \varphi_{\mathbf{i}} (x_{\varsigma,i_\varsigma} - \delta h)
                                                                     \end{array}
  \right\}.
  \end{eqnarray*}
 The above description can be derived as follows. Since the values $\varphi_{\mathbf{i}}$ and $\varphi_{\mathbf{i'}}$ only differ at the places $S := \{x_{\varsigma,i_\varsigma}, x_{\varsigma,i_\varsigma}-h, \ldots , x_{\varsigma,i_\varsigma} - nh\}$, a pair $(a,c)$ can only be an element of the difference of the sets $\calV (A,\varphi_{\mathbf{i}})$ and $\calV (A, \varphi_{\mathbf{i'}})$ if either $a-h \in S$, $a \in S$ or $c\in S$. The only case which occurs for $a-h \in S$ is $a = x_{\varsigma,i_\varsigma} - h$ as otherwise $\varphi_{\mathbf{i}}(a) - \varphi_{\mathbf{}}(a-h) = \varphi_{\mathbf{i'}}(a) - \varphi_{\mathbf{i'}}(a-h) =1$. For the same reason the case $a \in S$ reduces to $a = x_{\varsigma,i_\varsigma} - nh$. Also, we have by construction
 \begin{eqnarray*}
  \varphi_{\mathbf{i}}(a) &=& \varphi_{\mathbf{i}}(a+h)-1 \textnormal{ for all } a\in A, a\leq x_{\varsigma,i_\varsigma} \\
  \varphi_{\mathbf{i'}}(a) &=& \varphi_{\mathbf{i'}}(a+h)-1 \textnormal{ for all } a\in A, a < x_{\varsigma,i_\varsigma}.
 \end{eqnarray*}
  Thus if $(a,c) \in \calV(A,\varphi_{\mathbf{i}})$ or $(a,c) \in \calV(A,\varphi_{\mathbf{i'}})$ with $c \leq x_{\varsigma,i_\varsigma}$ (in particular if $c\in S$), then $a \in B$. Altogether, the elements $(a,c)$ of the difference of the sets $\calV (A,\varphi_{\mathbf{i}})$ and $\calV (A, \varphi_{\mathbf{i'}})$ satisfy either $a= x_{\varsigma,i_\varsigma} +h$, $a + x_{\varsigma,i_\varsigma}-nh$ or $(a,c) \in B\times S$. Since the transition from $\varphi_{\mathbf{i}}$ to $\varphi_{\mathbf{i'}}$ reduces the the values at $S$, the only cases which occur for $(a,c) \in \calV(A,\varphi_{\mathbf{i'}}) \setminus \calV(A,\varphi_{\mathbf{i}})$ are $a = x_{\varsigma,i_\varsigma}+h$ and $(a,c) \in B\times S$, which yields the decomposition $ \calV(A,\varphi_{\mathbf{i'}}) \setminus \calV(A,\varphi_{\mathbf{i}}) = D_1 \cup D_3$. For the same reason the only case which occurs for $(a,c) \in \calV(A,\varphi_{\mathbf{i}}) \setminus \calV(A,\varphi_{\mathbf{i'}})$ is $a = x_{\varsigma,i_\varsigma} - nh$, which gives the equality $\calV(A,\varphi_{\mathbf{i}}) \setminus \calV(A,\varphi_{\mathbf{i'}}) = D_2$.
  
  We obtain $|\calV(A,\varphi_{\mathbf{i'}})| - |\calV(A,\varphi_{\mathbf{i}})| = S_1 - S_2 + S_3$ with
  \begin{eqnarray*}
   S_1 &=& | \{a\in A;\, a > x_{\varsigma,i_\varsigma}+h, \varphi_{\mathbf{i}} (a) \in [\varphi_{\mathbf{i}}(x_{\varsigma,i_\varsigma})-\alpha+1, \varphi_{\mathbf{i}}(x_{\varsigma,i_\varsigma})]\} | \\
   S_2 &=& | \{a \in A;\, a > x_{\varsigma,i_\varsigma}-nh, \varphi_{\mathbf{i}} (a) \in [\varphi_{\mathbf{i}}(x_{\varsigma,i_\varsigma})-\alpha-n, \varphi_{\mathbf{i}}(x_{\varsigma,i_\varsigma})-n-1]\} | \\
   S_3 &=& | \{(b,\delta) \in B \times \{0,\ldots,n\};\,b \not=  x_{\varsigma,i_\varsigma}-nh, b < x_{\varsigma,i_\varsigma} - \delta h,\\
           & & \,\, \varphi_{\mathbf{i}}(b) \in    [\varphi_{\mathbf{i}}(x_{\varsigma,i_\varsigma})-\delta-\alpha+1, \varphi_{\mathbf{i}} (x_{\varsigma,i_\varsigma}) - \delta ] \} |.
  \end{eqnarray*}
 In order to calculate $S_1-S_2$, we introduce the sets $C_1$ and $C_2$, which are in some sense complementary to the sets considered for $S_1$ and $S_2$.
  \begin{eqnarray*}
   C_1 &=&  \{a\in A;\, a \leq x_{\varsigma,i_\varsigma}+h, \varphi_{\mathbf{i}} (a) \in [\varphi_{\mathbf{i}}(x_{\varsigma,i_\varsigma})-\alpha+1, \varphi_{\mathbf{i}}(x_{\varsigma,i_\varsigma})]\}  \\
   C_2 &=&  \{a \in A;\, a \leq x_{\varsigma,i_\varsigma}-nh, \varphi_{\mathbf{i}} (a) \in [\varphi_{\mathbf{i}}(x_{\varsigma,i_\varsigma})-\alpha-n, \varphi_{\mathbf{i}}(x_{\varsigma,i_\varsigma})-n-1]\}.
  \end{eqnarray*}
  As $\varphi_{\mathbf{i}} (x+h) = \varphi_{\mathbf{i}} (x) + 1$ for all $x\in A$ with $x \leq x_{\varsigma,i_\varsigma}$, we have $C_2 + (n+1)h \subset C_1$. We denote $C_3 := C_1 \setminus (C_2 + (n+1)h)$. Then
  \begin{eqnarray*}
   C_3 &=& \left\{ b+\delta h \in A;\, \begin{array}{l}
                                   b \in B, \delta \in \{0,\ldots, n \}, b \leq x_{\varsigma,i_\varsigma} + h - \delta h, \\ \varphi_{\mathbf{i}}(b) \in    [\varphi_{\mathbf{i}}(x_{\varsigma,i_\varsigma})-\delta-\alpha+1, \varphi_{\mathbf{i}} (x_{\varsigma,i_\varsigma}) - \delta ]
                                  \end{array}
  \right\} \\
   &=& \left\{ b+\delta h \in A;\, \begin{array}{l} b \in B \setminus \{x_{\varsigma,i_\varsigma}-nh\}, \delta \in \{0,\ldots, n\}, b \leq x_{\varsigma,i_\varsigma} + h - \delta h, \\ \varphi_{\mathbf{i}}(b) \in    [\varphi_{\mathbf{i}}(x_{\varsigma,i_\varsigma})-\delta-\alpha+1, \varphi_{\mathbf{i}} (x_{\varsigma,i_\varsigma}) - \delta ] \end{array} \right\} \\ & & \cup \bigg\{x_{\varsigma,i_\varsigma}-nh+\delta h;\, \delta = \max\{0,n+1-\alpha\}, \ldots, n+1 \bigg\} .
  \end{eqnarray*}
  In particular, we have $|C_3| \geq S_3 + \min \{\alpha,n+1\}$.

  Altogether, we get
  \begin{eqnarray*}
   \Delta &=& ( \sum_{k=0}^{\alpha-1} \sum_{l=0}^{n} |\{j; \mu^i_j = \varphi_{\mathbf{i}}(x_{\varsigma,i_\varsigma})-k-l\}|) - \min \{\alpha,n+1\} \\
   &=& \sum_{k=0}^{\alpha-1} \sum_{l=0}^n \big(\, | A_{(\varsigma)} \cap \varphi_{\mathbf{i}}^{-1} ( \{ \varphi_{\mathbf{i}}(x_{\varsigma,i_\varsigma})-k-l \}  )| - | A_{(\varsigma)} \cap \varphi_{\mathbf{i}}^{-1} ( \{ \varphi_{\mathbf{i}}(x_{\varsigma,i_\varsigma})-k-l-1\} ) | \,\big) \\ & &- \min \{ \alpha,n+1\} \\
   &=& | A_{(\varsigma)} \cap \varphi_{\mathbf{i}}^{-1} (\,[ \varphi_{\mathbf{i}}(x_{\varsigma,i_\varsigma}) - \alpha+1, \varphi_{\mathbf{i}}(x_{\varsigma,i_\varsigma})] \,) | - | A_{(\varsigma)} \cap \varphi_{\mathbf{i}}^{-1} (\, [\varphi_{\mathbf{i}}(x_{\varsigma,i_\varsigma})-\alpha-n, \varphi_{\mathbf{i}}(x_{\varsigma,i_\varsigma}) -n -1] \,) | \\ & & - \min \{ \alpha, n+1\} \\
   &=& (S_1 + |C_1|) - (S_2 + |C_2|) - \min \{ \alpha, n+1\}\\
   &=& S_1 -S_2 + |C_3| - \min \{ \alpha, n+1\} \\
   &\geq& S_1 - S_2 + S_3.
  \end{eqnarray*}

 \end{proof}

 \begin{proof}[Proof of Theorem \ref{thm main}]
  We reduced the claim of the theorem to $G = \Res_{k'/k} \GL_h$ and superbasic $b \in G(L)$ in Theorem \ref{thm reduction to superbasic}. In this case the  assertion is equivalent to $ \dim X_\mu (b) = \lengthG{\mu}{\nu}$ by Proposition \ref{prop superbasic}. As a consequence of Proposition \ref{prop decomposition} we get
  \[
   \dim X_\mu (b) = \max\{|\calV(A,\varphi)|;\, (A,\varphi) \textnormal{ is an extended EL-chart for } \mu \}.
  \]
  Now in Proposition \ref{prop upper bound V} we showed that the maximum is at least $\lengthG{\mu}{\nu}$ and in Theorem \ref{thm lower bound V} we showed that it is at most $\lengthG{\mu}{\nu}$, finishing the proof.
 \end{proof}

 \section{Irreducible components in the superbasic case}

 We now consider the $J_b(F)$-action on the irreducible components of $X_\mu(b)$ for superbasic $b$ and arbitrary $G$. Recall that the canonical projection $G \twoheadrightarrow G^{\ad}$ induces isomorphisms
 \[
  X_\mu(b)^\omega \cong X_{\mu_{\ad}}(b_{\ad})^{\omega_{\ad}}
 \]
 for all $X_\mu(b)^\omega \not= \emptyset$. As $J_b(F)$ acts transitively on the set of non-empty components $X_\mu(b)^\omega$ (cf.~\cite{CKV}~sect.3.3), this implies that the induced map on $J_b(F)$-orbits to $J_b^{\ad}(F)$-orbits of the respective irreducible components is bijective. Thus the set of irreducible components of $X_\mu(b)^\omega$ (resp.\ the set of $J_b(F)$-orbits of irreducible components of $X_\mu(b)$) only depends on the data $(G^\ad, \mu_\ad,b_\ad)$, so it suffices to consider the case $G = \Res_{k'/k} \GL_h$.

 \begin{lemma}
  Let $b \in G(L)$ be superbasic and $\omega \in \pi_1(G)$ such that $X_\mu(b)^\omega$ is non-empty. Then every $J_b(F)$-orbit of irreducible components of $X_\mu(b)$ contains a unique irreducible component of $X_\mu(b)^\omega$.
 \end{lemma}
 \begin{proof}
  Denote by $J_b(F)^0$ the stabiliser of $X_\mu(b)^\omega$ in $J_b(F)$. Using the argument above, it suffices to show that $J_b(F)^0$ stabilises the irreducible components of $X_\mu(b)$ in the case $G = \Res_{k'/k} \GL_h$.

  For this we consider the action of $J_b(F)$ on $N_{\hom}$. Let $g \in J_b(F)$ and let $v_0 := g(e_{0,1}), c(g) := \calI(v_0)$. Now every element $e_{\tau,i}$ can be written in the form $e_{\tau,i} = \frac{1}{t^j}(b\sigma)^k(e_{0,1})$ for some integers $j,k$; then
  \[
   g(e_{\tau,i}) = \frac{1}{t^j}(b\sigma)^k(v_0).
  \]
  Hence $\val\det (g) = (c(g)-1)_{\tau\in I}$, in particular we have $g \in J_b(F)^0$ if and only if $c(g) = 1$. In this case the above formula implies $\calI(g(e_{\tau,i})) = i_{(\tau)}$ and thus $\calI(g(v)) = \calI(v)$ for all $v \in N_{\hom}$.

  We obtain $A(M) = A(g(M))$ and $\varphi(M) = \varphi(g(M))$ for all $M \in X_\mu(b)^0$. Thus $\calS_{A,\varphi}$ is $J_b(F)^0$-stable for every extended EL-chart $(A,\varphi)$ for $\mu$. As the $\calS_{A,\varphi}$ are irreducible, every irreducible component of $X_\mu(b)^0$ is of the form $\overline{\calS_{A,\varphi}}$ and thus $J_b(F)^0$-stable.
 \end{proof}

 We denote by $\calM_\mu$ the set of extended EL-Charts $(A,\varphi)$ for $\mu$ for which $|\calV(A,\varphi)|$ is maximal. As noted above we have a bijection
 \begin{eqnarray*}
  \calM_\mu & \longleftrightarrow & \{ \textnormal{top-dimensional irreducible components of } X_\mu(b) \} \\
  (A,\varphi) & \longmapsto & \overline{\calS_{(A,\varphi)}}.
 \end{eqnarray*}
 Rapoport conjectured in \cite{rapoport06} that $X_\mu(b)$ is equidimensional. If this holds true, the above bijection becomes
 \begin{eqnarray*}
  \calM_\mu & \longleftrightarrow & \{ \textnormal{irreducible components of } X_\mu(b) \} \\
            & \longleftrightarrow & \{ J_b(F)-\textnormal{orbits of irreducible components of } X_\mu(b)\}.
 \end{eqnarray*}

 It is known that $J_b(F)$ does not act transitively on the irreducible components of $X_\mu(b)$ in general, even in the case $G = \GL_h$ (\cite{viehmann06}~Ex.~6.2). The following lemma shows that we have transitive action of $J_b(F)$ only in a few degenerate cases.

 \begin{lemma}
  \begin{enumerate}
   \item Assume there exist $\tau_1,\tau_2 \in I$ such that $\mu_{\tau_1}$ and $\mu_{\tau_2}$ are not of the form $(a,a,\ldots,a)$ for some integer $a$. Then $\calM_\mu$ contains more than one element.
   \item On the contrary, if there exists $\tau \in I$ such that $\mu_\varsigma = (a_\varsigma, a_\varsigma,\ldots,a_\varsigma)$ for some integer $a_\varsigma$ for all $\varsigma \not= \tau$ then 
   \[
    |\calM_\mu| = |\calM_{\mu_\tau}|
   \]
  \end{enumerate}
  where $\calM_{\mu_\tau}$ denotes the set of extended EL-charts for $\mu_\tau$ (in particular $d=1$ for these EL-charts).
 \end{lemma}
 \begin{proof}
  \emph{(1)} We assume without loss of generality that $\tau_1 \in [0,\tau_2)$. By Proposition~\ref{prop upper bound V} the cyclic EL-chart of type $\mu$ is contained in $\calM_\mu$. The same reasoning shows that $\calM_\mu$ also contains the cyclic EL-chart of type $(\mu_0',\ldots,\mu_{\tau_2-1}',\mu_{\tau_2},\ldots,\mu_{d-1})$ with $\mu'_\varsigma = (\mu_{\varsigma,h},\mu_{\varsigma,1},\ldots,\mu_{\varsigma,{h-1}})$. Note that our condition on $\mu_{\tau_1}$ implies $\mu_{\tau_1} \not= \mu'_{\tau_1}$ and thus $\mu \not= (\mu_0',\ldots,\mu_{\tau_2-1}',\mu_{\tau_2},\ldots,\mu_{d-1})$. \smallskip \\
  \emph{(2)} The claim holds as the bijection
  \begin{eqnarray*}
   \{ \textnormal{extended EL-charts for } \mu\} & \longleftrightarrow & \{ \textnormal{extended semi-modules for } \mu_\tau\} \\
   (A,\varphi) & \longmapsto & (A_\tau, \varphi_\tau) \\
   (\sqcup_{\tau\in I} B, \sqcup \varphi) & \longmapsfrom & (B,\varphi)
  \end{eqnarray*}
  preserves $|\calV|$. 
 \end{proof}

 We now consider the case $\mu$ minuscule. De Jong and Oort showed that in this case, $|\calM_\mu| = 1$ for (extended) semi-modules (i.e.\ the case $G = \GL_n$). For EL-charts we have the following conjecture, which generalises \cite{dJO00}~Rem.~6.16.
 
 \begin{conjecture}
  Let $\mu$ be minuscule. Then the construction of $\tilde\mu$ in the proof of Thm.~\ref{thm lower bound V} induces a bijection
  \[
   \{\textnormal{(extended) EL-charts for }\mu\} \rightarrow \{\tilde\mu \in W.\mu;\, \nu \preceq \bar{\tilde\mu}\}
  \]
  where $W = (S_n)^I$ denotes the absolute Weyl group of $G$. In particular
  \begin{eqnarray*}
   |\calM_\mu| &=& |\{\tilde\mu \in W.\mu;\, \nu \preceq \bar{\tilde\mu}, \lengthG{\nu}{\tilde\mu} = 0\}| \\
               &=& |\{\tilde\mu \in W.\mu;\, \orbitsum{\tilde\mu} = (\underbrace{\lfloor \frac{m}{h} \rfloor,\ldots,\lfloor \frac{m}{h} \rfloor}_{h\cdot(1-\{m/h\})},\underbrace{\lceil \frac{m}{h} \rceil, \ldots, \lceil \frac{m}{h} \rceil}_{h\cdot\{m/h\}})\}|,
  \end{eqnarray*}
  Here $m$ is defined as in section~\ref{sect notation}, i.e.\ $\nu = (\frac{m}{dh},\ldots,\frac{m}{dh})$.
 \end{conjecture}

\providecommand{\bysame}{\leavevmode\hbox to3em{\hrulefill}\thinspace}

\end{document}